\documentclass[final]{siamltex}

\usepackage{amsmath,amssymb}

\newcommand\scalemath[2]{\scalebox{#1}{\mbox{\ensuremath{\displaystyle #2}}}}
\DeclareMathOperator*{\argmax}{arg\,max}
\usepackage{listings,url,verbatim}    

\usepackage{framed}
\usepackage{subfigure}
\usepackage{longtable}
\usepackage{lscape}
\usepackage{multirow}
\usepackage[pdftex]{graphicx}
\usepackage{rotating}

\DeclareFontFamily{OT1}{pzc}{}
\DeclareFontShape{OT1}{pzc}{m}{it}{<-> s * [0.900] pzcmi7t}{}
\DeclareMathAlphabet{\mathpzc}{OT1}{pzc}{m}{it}

\newtheorem{algorithm}{Algorithm}[section] 
 
\newtheorem{remark}{Remark}[section] 
\newenvironment{REMARK}{\begin{remark}\rm}{\end{remark}} 
\newtheorem{definition1}{Definition}[section] 
 
\newtheorem{example}{Example}[section] 
\newenvironment{EXAMPLE}{\begin{example}\rm}{\end{example}} 

\usepackage{framed}
\usepackage{subfigure}
\usepackage{longtable}
\usepackage{lscape}
\usepackage{multirow}
\usepackage[pdftex]{graphicx}
\usepackage{rotating}

\DeclareFontFamily{OT1}{pzc}{}
\DeclareFontShape{OT1}{pzc}{m}{it}{<-> s * [0.900] pzcmi7t}{}
\DeclareMathAlphabet{\mathpzc}{OT1}{pzc}{m}{it}

\makeatletter
\def\munderbar#1{\underline{\sbox\tw@{$#1$}\dp\tw@\z@\box\tw@}}
\makeatother

\def\headerrule{\rule[.01in]{\textwidth}{.01in}}

\newenvironment{INDENT}{\let\\=\item \parskip .2em
  \list{}{\topsep 0em \partopsep 0em
          \parsep\parskip \itemsep 0em
          \leftmargin 1.8em \rightmargin 0em}\item[]}{\endlist}
\def\IF/{{\bf if}}
\def\THEN/{{\bf then}}
\def\FOR/{{\bf for}}
\def\TO/{{\bf to}}
\def\STEP/{{\bf step}}
\def\DO/{{\bf do}\begin{INDENT}}
\def\ENDFOR/{\end{INDENT}{\bf endfor}}
\def\ENDIF/{\end{INDENT}{\bf endif}}
\def\WHILE/{{\bf while}}
\def\ENDWHILE/{\end{INDENT}{\bf endwhile}}
\def\REPEAT/{{\bf repeat}\begin{INDENT}}
\def\UNTIL/{\end{INDENT}{\bf until}}
\def\BEGIN/{\begin{INDENT}{\bf begin}\begin{INDENT}}
\def\END/{\end{INDENT}{\bf end}\end{INDENT}}

\newcommand{\ben}{\begin{enumerate}}
\newcommand{\een}{\end{enumerate}}
\newcommand{\be}{\begin{equation}}
\newcommand{\ee}{\end{equation}}

\title{Gaussian Elimination with Randomized Complete Pivoting}

\author{Christopher Melgaard\thanks{Department of Mathematics, University of California, Berkeley.}
\and Ming Gu\thanks{Department of Mathematics, University of California, Berkeley. This research was supported in part by NSF Award CCF-1319312.}}
\begin{document}

\DeclareGraphicsExtensions{.pdf,.png,.jpg,.eps}

\maketitle

\begin{abstract} Gaussian elimination with partial pivoting (GEPP) has long been among the most widely used methods for computing the LU factorization of a given matrix. However, this method is also known to fail for matrices that induce large element growth during the factorization process. In this paper, we propose a new scheme, Gaussian elimination with randomized complete pivoting (GERCP) for the efficient and reliable LU factorization of a given matrix. GERCP satisfies GECP (Gaussian elimination with complete pivoting) style element growth bounds with high probability, yet costs only marginally higher than GEPP. Our numerical experimental results strongly suggest that GERCP is as reliable as GECP and as efficient as GEPP for computing the LU factorization. 
\end{abstract}
 
\section{Background and Motivation}
Solving linear systems of equations
\begin{equation}\label{Eqn:eqn}
{\bf A}{\bf x} = {\bf b}, 
\end{equation}
where ${\bf A} \in \mathbb{R}^{n\times n}$ and ${\bf x, b} \in \mathbb{R}^{n}$, is a fundamental problem in numerical linear algebra and scientific computing. Gaussian Elimination with Partial Pivoting (GEPP) solves this problem by computing the LU factorization of ${\bf A}$ and is typically efficient and reliable. Over the years, GEPP has been repeatedly re-designed and re-implemented for better performance, and is the backbone for generations of mathematical software packages, including LINPACK \cite{linpack}, LAPACK \cite{LAPACK}, PLAPACK, SCALAPACK, PLASMA and MAGMA. GEPP routines in today's mathematical software libraries such as the Intel {\tt mkl} are capable of solving linear systems of equations with tens of thousands of variables at or near the peak of the machine's speed. 

Efficiency aside, an equally important consideration is numerical reliability. While algorithms for solving eigenvalue problems have become significantly more stable over the years, GEPP was known to be, and remains, a method that is mostly stable in practice but unstable for many well-known matrices including some from common integral equations and differential equations applications \cite{FOSTER4,wrightfail}. 

Pivoting plays a crucial role in the reliability of Gaussian elimination (GE), which is tied to \emph{element growth} within the LU factorization process. The most naive version of GE, Gaussian elimination without pivoting (GENP), does not perform any pivoting and only requires $\frac{2}{3} n^3 + O(n^2)$ floating point operations with no entry comparisons \cite{demmelbk}. However, this method can suffer from uncontrolled element growth and is only known to be reliable in a few instances like diagonally dominant matrices among others.  The most popular version of GE is GEPP, which limits element growth to at most exponential by swapping the rows of ${\bf A}$ (i.e., partial pivoting) during elimination, and is numerically stable on average. The additional cost, about $\frac{1}{2}n^2$ entry comparisons and the associated data movement, is typically a small fraction of the total GE cost. The most reliable version of GE is Gaussian elimination with complete pivoting (GECP), which swaps both rows and columns for sub-exponential element growth \cite{wilk1} and is universally believed to be always backward stable in practice \cite{demmelbk}. However, GECP is prohibitively slow with $\frac{1}{3}n^3 + O(n^2)$ entry comparisons and relatively little memory reuse \cite{demmelbk, highambk}.

Rook pivoting \cite{foster3,geogauss2,rookpoole} is an attempt to speedup complete pivoting while maintaining the guarantee of sub-exponential element growth. Rook pivoting is part of the {\tt LUSOL} package \cite{LUSOL} for sparse LU factorization. Despite having better performance in the ``average'' case, there are many matrices that still require $O(n^3)$ entry comparisons in the worst case, providing a negligible speedup over complete pivoting \cite{highambk}. 

In this paper, we propose a novel pivoting scheme called \emph{Gaussian elimination with randomized complete pivoting} (GERCP). We show that GERCP satisfies a stability condition similar to that of complete pivoting, suggesting that these methods share similar stability properties. Yet, we also demonstrate that GERCP is comparable to GEPP in computational overhead. Our numerical experimental results strongly suggest that GERCP is a numerically stable and computationally efficient alternative to GEPP.

Randomization has been used to fix the numerical instability of GEPP in the literature, through GE on the product of random matrices and ${\bf A}$ to avoid catastrophically bad pivots. These methods are known to work well in practice in general, but they still lack effective control on element growth, and can be much less accurate than GEPP. 

In Section \ref{Sec:Notation}, we introduce the necessary notation and background for the paper. In Section \ref{Sec:GERCP}, we introduce GERCP and state/prove some important properties. In section $4$, we talk about numerical experiements and implementations of GERCP. The appendix has results needed by the proofs in section $3$.

\section{The Setup and Background}\label{Sec:Notation}
In this paper, we consider Gaussian elimination on an invertible square matrix ${\bf A}  \in \mathbb{R}^{n\times n}$, although our algorithms and analysis carry over to the cases of singular matrices and rectangular matrices with few modifications.

\subsection{Notation}

We will follow the familiar slight abuse of notation from scientific computing and numerical linear algebra, mimicking the way that LAPACK overwrites the input matrix with the $L$ and $U$ factors. The diagonal of ${\bf A}$ becomes the diagonal of $U$ because the diagonal of $L$ is always $1$ and thus does not need to be stored.

\begin{algorithm}\label{Alg:gematlab} {\bf Classical Gaussian Elimination in Matlab Notation} \\
\headerrule

\begin{tabular}{ll}
{\bf Input:} & $n \times n$ matrix ${\bf A}$ \\
{\bf Output:} & lower triangular $L$ with unit diagonal, upper triangular $U$, \\
& row permutation $\Pi_r$, column permutation $\Pi_c$.  \\
\end{tabular}

\noindent \headerrule

\begin{INDENT}
{\noindent \bf set} $A = {\bf A}$ \\
{\noindent \bf for} $k = 1, \cdots, n - 1$ {\bf do}   (i.e. called $k^{th}$ stage of LU) \\
\begin{INDENT}
\begin{enumerate}
\item {\bf select} column pivot (INSERT PIVOTING RULE). \\
 {\bf swap} (UPDATE $A$ AND $\Pi_c$ WITH PIVOT DECISION).
\item {\bf select} row pivot (INSERT PIVOTING RULE). \\
 {\bf swap} (UPDATE $A$ AND $\Pi_r$ WITH PIVOT DECISION).
\item {\bf compute} $A(k+1:n,k) = A(k+1:n,k)/A(k,k)$;
\item {\bf compute} $A(k+1:n,k+1:n) = A(k+1:n,k+1:n) - A(k+1:n,k) * A(k,k+1:n) $;
\end{enumerate}
\end{INDENT}
\end{INDENT}
\headerrule
\end{algorithm}

\begin{REMARK} While the working matrix $A$ has been overwritten in Algorithm \ref{Alg:gematlab}, in our subsequent discussions we will refer $L$ and $U$ as the triangular matrices stored in $A$ and still refer ${\bf A}$ as the original input matrix. We will use $A_k\in \mathbb{R}^{n\times n}$ to refer explicitly to the working matrix before the $k^{th}$ stage of the outer most loop. Thus, $A_{n}$ refers to the working matrix after the algorithm terminates, i.e. {\emph after} the $(n-1)^{th}$ stage of the outer loop.
\end{REMARK}

\begin{REMARK} For easy of discussion, we have written Algorithm \ref{Alg:gematlab} in such a way that, for each $k$,  it performs possible column pivoting before any possible row pivoting. GENP, GEPP, GECP and rook pivoting can all be written in this form. 
\end{REMARK}

For any appropriate dimension $m$, we denote ${\bf e}_i \in \mathbb{R}^{m}$ to be the $i^{th}$ standard basis vector, i.e. a vector with all entries equal to $0$ except for the $i^{th}$ entry which equals $1$; we also denote ${\bf e} \in \mathbb{R}^{m}$ to be the vector with all entries equal to $1$. Any permutation matrix $\Pi \in \mathbb{R}^{n \times n}$ is a square matrix with exactly one entry equal to $1$ in each row and column, and all other entries equal to $0$. We refer to the permutation induced by $\Pi$ as $\pi : \{ 1, \cdots, n \} \rightarrow \{1, \cdots, n \}$ in the sense that $\pi(i) = j$ if and only if $\Pi {\bf e}_i = {\bf e}_j$. We will commonly make use of the swap or \emph{$2$-cycle} permutation given by $\pi_{(i,j)}$ or $\Pi_{(i,j)}$ in matrix form defined by
\[ \pi_{(i,j)} (i) = j, \hspace{1cm} \pi_{(i,j)} (j) = i \hspace{0.5cm} \text{ and } \hspace{0.5cm} \pi_{(i,j)} (k) = k \text{, for all $k \neq i,j$}  \]
 We denote the final row and column permutations of an algorithm as $\Pi_r$ and $\Pi_c$ respectively. At the $k^{th}$ stage of LU, Algorithm \ref{Alg:gematlab} will swap the $k^{th}$ column with the $\alpha_k^{th}$ column and the $k^{th}$ row with the $\beta_k^{th}$ row. As a result, we can write 
\[{\displaystyle 
\Pi_c =\Pi_{(n-1,\alpha_{n-1})} \cdots \Pi_{(2,\alpha_{2})} \Pi_{(1,\alpha_{1})} , \quad 
\Pi_r = \Pi_{(n-1,\beta_{n-1})} \cdots \Pi_{(2,\beta_{2})} \Pi_{(1,\beta_{1})}} \]

as a product of the individual column/row swaps. Furthermore, we define the next notation to give us the first $k-1$ swaps and the last $n-k$ swaps
\begin{align*}
\Pi_{c,k} &= \Pi_{(k-1,\alpha_{k-1})} \cdots \Pi_{(2,\alpha_{2})} \Pi_{(1,\alpha_{1})} \\
\Pi_{c,-k} &= \Pi_{(n-1,\alpha_{n-1})} \cdots \Pi_{(k+1,\alpha_{k+1})} \Pi_{(k,\alpha_{k})}
\end{align*}
Also, we will use the analogous definition for $\Pi_{r,k}$ and $\Pi_{r,-k}$.


In Figure \ref{fig:Matlab}, we describe the use of the MATLAB colon notation in combination with the permutations above to explain our rows/columns reorderings of a matrix and its select submatrices. 

\begin{figure}[t!] \label{fig:Matlab}
	\centering
	\begin{tabular}{ |l|l|l|p{5.3cm}| }
	\hline
	\multicolumn{4}{ |c| }{{\bf \large Common examples of Matlab notation} for $1\leq i\leq p\leq m$ and $1\leq j\leq q\leq n$} \\
	\hline
	{\bf Notation} & {\bf Pivoted Notation} & {\bf Dimensions} & {\bf Description} \\ \hline
	$B(:,:)$ & $B(\pi_1(:),\pi_2(:))$ & $\mathbb{R}^{n \times n}$ & Entire matrix $B$ or $\Pi_1B\Pi_2^T$ resp. \\ \hline
	$B(i,:)$ & $B(\pi_1(i),\pi_2(:))$ & row vector in $\mathbb{R}^{n}$ & $i^{th}$ row of $B$ or $\Pi_1B\Pi_2^T$ resp. \\ \hline
	$B(:,j)$ & $B(\pi_1(:),\pi_2(j))$ & column vector in $\mathbb{R}^{n}$ & $j^{th}$ column of $B$ or $\Pi_1B\Pi_2^T$ resp. \\ \hline
	$B(i,j:q)$ & $B(\pi_1(i),\pi_2(j:q))$& row vector in $\mathbb{R}^{q-j+1}$ & $j^{th}$ through $q^{th}$ entries of $i^{th}$ row of $B$ or $\Pi_1B\Pi_2^T$ resp. \\ \hline
	$B(i:p,j)$ & $B(\pi_1(i:p),\pi_2(j))$ & column vector in $\mathbb{R}^{p-i+1}$ & $i^{th}$ through $p^{th}$ entries of $j^{th}$ column of $B$ or $\Pi_1B\Pi_2^T$ resp.  \\ \hline
	$B(i:p,j:q)$ & $B(\pi_1(i:p),\pi_2(j:q))$ & $\mathbb{R}^{(p-i+1) \times (q-j+1)}$ & Submatrix from intersection $i^{th}$ through $p^{th}$ rows and $j^{th}$ through $q^{th}$ columns of $B$ or $\Pi_1B\Pi_2^T$ resp. \\
\hline
	\end{tabular}
	\caption{Table of Matlab notations}
\end{figure}

Let $\pi_c$ and $\pi_r$ be the permutation of columns and rows performed by LU respectively. We use the following notation to refer to matrices with the final pivoting applied apriori
\begin{align*}
{\bf A}^{\Pi_c} &= \Pi_r A \Pi_c^T = A(\pi_r(:),\pi_c(:)) \\
A_k^{\Pi_c} &= \Pi_{r,-k} A \Pi_{c,-k}^T = A_k(\pi_{r,-k}(:),\pi_{c,-k}(:))
\end{align*} 
We do this because all of the pivoting methods discussed in this paper are \emph{top-heavy} as defined in Definition \ref{topheavy}. The row pivots $\pi_r$ of a top-heavy pivoting strategy are deterministic given the column pivots $\pi_c$ applied to ${\bf A}$ by the LU factorization because each row pivot must satisfy equation \eqref{eqn:topheavy}. Therefore, when writing ${\bf A}^{\Pi_c}$, it is understood that the row pivots are the unique set of top-heavy row pivots. 

Schur complements form a crucial role in Gaussian elimination. We establish notation for Schur complements as $S_k \in \mathbb{R}^{(k:n)\times (k:n)}$. Notice the use of Matlab notation $(k:n)\times (k:n)$ instead of $(n-k+1) \times (n-k+1)$. The Schur complement $S_k$ will act like a normal $(n-k+1) \times (n-k+1)$ matrix for most operations like matrix multiplication, matrix addition and ect. However, when using the Matlab notation in Figure \ref{fig:Matlab} to access entries of $S_k$, we impose the abuse of notation that rows and columns are enumerated from $k$ to $n$, instead of $1$ to $n-k+1$. For example, \begin{enumerate}
\item[$\bullet$] top left entry of $S_k$ is denoted as $S_k(k,k)$, but NOT $S_k(1,1)$. 
\item[$\bullet$] submatrix of last two columns of $S_k$ is denoted by $S_k(:,n-1:n)$, but NOT $S_k(:,n-k:n-k+1)$. \end{enumerate}
This makes our analysis much cleaner and more straightforward because it synchronizes the enumeration of columns/rows between the $k^{th}$ working matrix $A_k$ and the $k^{th}$ Schur complement $S_k$ which is a submatrix for $A_k$, i.e. $S_k(k:n,k:n) = A_k(k:n,k:n)$. Given this notation for Schur complements, we formally define the working Schur complement at the $k^{th}$ stage $S_k$ and the fully pivoted $k^{th}$ Schur complement $S_k^{\Pi_c}$
\begin{align*}
S_k(k:n,k:n) &= A_k(k:n,k:n) \\
S_k^{\Pi_c}(k:n,k:n) &= A_k^{\Pi_c} (k:n,k:n)
\end{align*}
where $A_k\in\mathbb{R}^{n\times n}$ is the working matrix at the $k^{th}$ stage. This implies that $S_k^{\Pi_c}(k:n,k:n) = S_k(\Pi_{r,-k}(k:n),\Pi_{c,-k}(k:n))$, i.e. $S_k$ has the pivots only up to the $k^{th}$ step and $S_k^{\Pi}$ is already pivoted into the final permutation so that no further pivots are required.
  
\begin{figure}[t!]
	\centering
	\begin{tabular}{ |l|l|p{5.6cm}|p{4cm}| }
	\hline
	\multicolumn{4}{ |c| }{\bf \large Different $A, L$ and $U$ notations} \\
	\hline
	{\bf Notation} & {\bf Dimensions} & {\bf Description} & {\bf Algorithm Pivots} \\ \hline
	${\bf A}$, $A_1$, $S_1$ & $\mathbb{R}^{n \times n}$ & Unadulterated input matrix & None \\ \hline
	${\bf A}^{\Pi_c}$,$A_1^{\Pi_c}$, $S_1^{\Pi_c}$ & $\mathbb{R}^{n\times n}$ & Pivoted input matrix & All pivots applied apriori \\ \hline
	$A$ & $\mathbb{R}^{n\times n}$ & Current working matrix. Overwrite triangular $L, U$ factors and Schur complement in place & Pivots applied as determined by algorithm \\ \hline
	$A_k$ & $\mathbb{R}^{n\times n}$ & Working matrix at $k^{th}$ stage. Overwrite triangular $L_k, U_k$ factors and Schur complement $S_k$ in place & Pivots applied as determined by algorithm \\ \hline
	$A_k^{\Pi_c}$ & $\mathbb{R}^{n\times n}$ & Working matrix at $k^{th}$ stage, i.e. $\Pi_{r,-k} A_k \Pi_{c,-k}^T$ & All pivots applied apriori \\ \hline
	$S_k$ & $\mathbb{R}^{(k:n) \times (k:n)}$ & $k^{th}$ Schur Complement of $\Pi_{r,k}{\bf A}\Pi_{c,k}^T$ & Pivots applied as determined by algorithm \\ \hline
	$S_k^{\Pi_c}$ & $\mathbb{R}^{(k:n) \times (k:n)}$ & $k^{th}$ Schur Complement of $\Pi_r{\bf A}\Pi_c^T$ & All pivots applied apriori \\
\hline
	\end{tabular}
	\caption{Table of Notations}
\end{figure}
The unit lower triangular matrix $L_{k+1}\in\mathbb{R}^{n\times k}$ and the upper triangular matrix $U_{k+1}\in\mathbb{R}^{k\times n}$ relate to the working matrices $A_{k+1}$ 
\begin{align*}
U_{k+1} &= \left( \begin{array}{cccccc} A_{k+1}(1,1) & A_{k+1}(1,2) & \cdots & A_{k+1}(1,k) & \cdots & A_{k+1}(1,n) \\  & A_{k+1}(2,2) &\cdots & A_{k+1}(2,k) & \cdots & A_{k+1}(2,n) \\  &  & \ddots & \vdots & \cdots & \vdots \\  &  &  & A_{k+1}(k,k) & \cdots & A_{k+1}(k,n) \end{array} \right) \\
L_{k+1} &= \left( \begin{array}{cccc} 1 &  &  &  \\ A_{k+1}(2,1) & 1 &  &  \\ \vdots & \vdots & \ddots &  \\ A_{k+1}(k,1) & A_{k+1}(k,2) & \cdots & 1 \\ \hspace{0.37cm}A_{k+1}(k\hspace{-0.1cm}+\hspace{-0.1cm}1,1) & \hspace{0.37cm}A_{k+1}(k\hspace{-0.1cm}+\hspace{-0.1cm}1,2) & \cdots & \hspace{0.37cm}A_{k+1}(k\hspace{-0.1cm}+\hspace{-0.1cm}1,k) \\ \vdots & \vdots & \vdots & \vdots  \\ A_{k+1}(n,1) & A_{k+1}(n,2) & \cdots & A_{k+1}(n,k) \end{array} \right) 
\end{align*}
for $1<k\leq n$. Remember that $L=L_{n+1}\in\mathbb{R}^{n\times n}$ and $U=U_{n+1}\in\mathbb{R}^{n\times n}$. 

\subsection{Useful Preliminary Tools}
\subsubsection{Tools from Matrix Analysis} In finite dimensions, it is well known that all normed vector spaces are topologically equivalent. The following lemma gives the equivalence between the $\| \cdot \|_2$ and $\| \cdot \|_{\infty}$ vector norms, which is easily proven from the definitions of the norms. This result can also be seen as a trivial application of the Fritz John Ellipsoid Theorem \cite{john2014extremum} for convex regions that are symmetric about the origin. 
\begin{lemma}[Equivalence of $\ell_2$ and $\ell_{\infty}$ in finite dimensions] \label{inf_to_2}
Let ${\bf x} \in \mathbb{R}^d$. Then
\[ \frac{1}{\sqrt{d}}\left\| {\bf x} \right\|_{2} \leq \left\| {\bf x} \right\|_{\infty} \leq \left\| {\bf x} \right\|_{2} \]
\end{lemma}

\begin{definition}[Subordinate matrix norms \cite{highambk,horn}]
Let $1 \leq p,q \leq \infty$ and let ${\bf B}  \in \mathbb{R}^{m \times n}$. A {\bf subordinate matrix norm} is a matrix norm of the form
\[ \left\|{\bf B}\right\|_{q,p} = \sup_{0 \neq {\bf x} \in \mathbb{R}^n} \frac{\left\|{\bf B}{\bf x} \right\|_{p}}{\left\| {\bf x} \right\|_{q}} = \max_{\left\| {\bf x} \right\|_{q} = 1} \left\|{\bf B}{\bf x} \right\|_{p} \]
\end{definition}
Note that for normal matrix operator norms, we have that $\|{\bf B} \|_{p} = \|{\bf B} \|_{p,p}$. We make use of a little known subordinate matrix norm by setting $q=1$ in the above as in exercise 6.11 of \cite{highambk}.
\begin{lemma}[Maximum $\ell_p$ column norm $\left\| \cdot \right\|_{1, p}$]
Let ${\bf B} \in \mathbb{R}^{m\times n}$. Then, we have
\[ \left\|{\bf B}\right\|_{1,p} = \max_{1 \leq i \leq n} \left\|{\bf B}{\bf e}_i \right\|_p = \max_{1 \leq i \leq n} \left\| {\bf B}(:,i) \right\|_p\]
\end{lemma}
\begin{proof}
Let $0 \neq {\bf x} = \left[ \begin{array}{c} x_1 \\ \vdots \\ x_n \end{array} \right] \in \mathbb{R}^n$. Then
\begin{align*}
\left\|{\bf B}{\bf x} \right\|_p &= \left\| \sum_{i=1}^n x_i {\bf B}(:,i) \right\|_p \leq \sum_{i=1}^n \left| x_i \right| \left\| {\bf B}(:,i) \right\|_p \leq \left( \max_{j} \left\| {\bf B}(:,j) \right\|_p \right) \sum_{i=1}^n \left| x_i \right| = \left( \max_{j} \left\| {\bf B}(:,j) \right\|_p \right) \left\| {\bf x} \right\|_1
\end{align*}
Dividing both sides by $\left\| {\bf x} \right\|_1$ and taking a supremum, we arrive at
\[ \left\|{\bf B}\right\|_{1,p} \stackrel{def}{=} \sup_{0\neq {\bf x} \in \mathbb{R}^n} \frac{\left\|{\bf B}{\bf x} \right\|_p}{\left\| {\bf x} \right\|_1} \leq \max_{j} \left\| {\bf B}(:,j) \right\|_p \]
Also, let $j^*= \argmax_j \left\| {\bf B}(:,j) \right\|_p$ and we arrive at our conclusion by observing
\[ \max_{j} \left\| {\bf B}(:,j) \right\|_p = \frac{\left\|{\bf B}{\bf e}_{j^*} \right\|_p}{\left\| {\bf e}_{j^*} \right\|_1} \leq \sup_{0\neq {\bf x} \in \mathbb{R}^n} \frac{\left\|{\bf B}{\bf x} \right\|_p}{\left\| {\bf x} \right\|_1} \stackrel{def}{=} \left\|{\bf B}\right\|_{1,p} \] \end{proof} 

In particular, we make frequent use of the two following subordinate matrix norms:
\begin{align}
\left\|{\bf B}\right\|_{1, \infty} &= \max_{1 \leq j \leq n} \left\|{\bf B}{\bf e}_j \right\|_{\infty} = \max_{1\leq i,j \leq n} \left| {\bf B}(i,j) \right| \hspace{1.5cm} \text{({\bf Maximum entry norm})} \label{maxentrynorm} \\
\left\|{\bf B}\right\|_{1 , 2} &= \max_{1 \leq i \leq n} \left\|{\bf B}{\bf e}_j \right\|_2 = \max_{1 \leq j \leq n} \left\| {\bf B}(:,j) \right\|_2 \hspace{1.5cm} \text{({\bf Maximum $\ell_2$ column norm})} \label{l2colnorm}
\end{align}

The next lemma, from line $(6.19)$ in chapter 6.3 of \cite{highambk}, will allow us to control each operator norm $\| \cdot \|_p$ for $1\leq p\leq \infty$ of our residual error via the largest absolute column and row sum of the residual error.
\begin{theorem}[Special case of Riesz-Thorin theorem \cite{highambk}] \label{2_from_inf_and_1}
Let ${\bf B}  \in \mathbb{R}^{m\times n}$ and let $1\leq p\leq\infty$. Then
\[ \left\|{\bf B}\right\|_p \leq \left\|{\bf B}\right\|_1^{\frac{1}{p}} \left\|{\bf B}\right\|_{\infty}^{1-\frac{1}{p}} \]
\end{theorem} 
Setting $p=2$ in the above theorem, we get a bound on the spectral norm of $B$. The volume of a parallelepiped (parallelotope) formed from the columns of a matrix ${\bf B}$ is given by the absolute value of the determinant of $B$. A rectangle is formed by forcing the parallelepiped to have only right angles between each of its vectors. The next result states that the volume of parallelepiped (parallelotope) is bounded above by that of the corresponding rectangle (hyperrectangle). This result is crucial to deriving element and column growth factors.
\begin{theorem}[Hadamard's Inequality \cite{horn}] \label{hadamard} 
Let ${\bf B}  \in \mathbb{R}^{m \times m}$. Then, we have that
\[ \left| \det({\bf B}) \right| \leq \prod_{j=1}^{m} \left\| {\bf B}(:,j) \right\|_2 \]
\end{theorem}

\subsubsection{Tools from Probability Theory}
The union bound is a basic yet important result whose proof is typically left as an exercise in most introductory probability textbooks. We will make use of it in combination with the famous De Morgan laws to consider the probability of an intersection of highly coupled events in the analysis of GERCP.
\begin{lemma}[Union Bound or Boole's Inequality] \label{unionb}
For events $E_1, E_2, ..., E_m$, we have that
\[ \mathbb{P} \left( \bigcup_{i = 1}^{m} E_i \right) \leq \sum_{i=1}^{m} \mathbb{P} \left( E_i \right) \]
\end{lemma}

As GERCP is a randomized algorithm, the factorization it produces will be random. The Law of Total Probability below is the basis on which we analyze the reliability of GERCP regardless of the column and row permutations chosen by GERCP.
\begin{theorem}[Law of Total Probability]\label{totalprob}
Given $m$ mutually exclusive events $E_1, \cdots, E_m$ whose probabilities sum to unity, then
\[{\displaystyle  \mathbb{P} \left(B\right) =  \sum_{i=1}^{m} \mathbb{P} \left(B | E_i\right) \mathbb{P} \left( E_i\right), } \]
where $B$ is an arbitrary event, and $\mathbb{P} \left(B | E_i\right)$  is the conditional probability of $B$ assuming $E_i$.
\end{theorem}

We will also need the following large deviation result, due to Johnson and Lindenstrauss, from probability theory and theoretical computer science. Theorem \ref{randproj} has been the main theoretical foundation in the recent development of randomized algorithms in numerical linear algebra and data analysis. It will also be the primary theoretical justification for the reliability of GERCP. 

\begin{theorem} [Random Projection Method (Johnson-Lindenstrauss) \cite{vempala}] \label{randproj} 
Let ${\bf x} \in \mathbb{R}^d$ and $\epsilon > 0$. Assume that the entries in $\Omega \in \mathbb{R}^{r \times d}$ are sampled independently from $\mathcal{N}(0,1)$. Then
\begin{align*}
\mathbb{P}\left( \left(1-\epsilon \right) \left\| {\bf x} \right\|_2^2 \geq \left\| \frac{1}{\sqrt{r}} \Omega {\bf x} \right\|_2^2 \right) &\leq \exp \left(-\frac{(\epsilon^2 - \epsilon^3)r}{4}\right) \\
\mathbb{P}\left( \left\| \frac{1}{\sqrt{r}} \Omega {\bf x} \right\|_2^2 \geq (1+\epsilon) \left\| {\bf x} \right\|_2^2 \right) &\leq \exp \left(-\frac{(\epsilon^2 - \epsilon^3)r}{4}\right)
\end{align*}
and
\begin{equation}\label{Eqn:randproj}
{\displaystyle\mathbb{P}\left( \left(1-\epsilon \right) \left\| {\bf x} \right\|_2^2 \leq \left\| \frac{1}{\sqrt{r}} \Omega {\bf x} \right\|_2^2 \leq (1+\epsilon) \left\| {\bf x} \right\|_2^2 \right) \geq 1 - 2 \exp \left(-\frac{(\epsilon^2 - \epsilon^3)r}{4}\right).}
\end{equation}
\end{theorem}

Due to the central importance of Theorem \ref{randproj} in GERCP, we make the next definition
\begin{definition} \label{Def:JLvec} A given vector ${\bf x} \in \mathbb{R}^d$ satisfies the $\epsilon-$JL condition under random mapping $\Omega$ if 
\[{\displaystyle \sqrt{1-\epsilon } \left\| {\bf x} \right\|_2 \leq \left\| \frac{1}{\sqrt{r}} \Omega {\bf x} \right\|_2 \leq \sqrt{1+\epsilon }\left\| {\bf x} \right\|_2}. \]
\end{definition}

\begin{REMARK} Despite its simplicity, Theorem \ref{randproj} asserts the surprisingly strong norm-preserving abilities under a random projection. For any given $\Delta \in (0,1)$, $x$ satisfies the $\epsilon-$JL condition under random mapping $\Omega$ with probability at least $1-\Delta$ for any 
\begin{equation}\label{Eqn:r}
{\displaystyle r \geq  \frac{4}{\epsilon^2 - \epsilon^3} \log\left( \frac{2}{\Delta}\right).}
\end{equation}

\noindent In particular, for $\epsilon = \frac{1}{2}$ and $\Delta = 10^{-5}$, $r = 400$ satisfies equation~\eqref{Eqn:r}, regardless of $d$. In practice, however, one can typically choose a much smaller value of $r$ for $x$ to satisfy $\epsilon-$JL condition.
\end{REMARK}

\subsection{Numerical Error and Stability of LU Factorization}
With computer roundoff, Theorem 9.3 of \cite{highambk} gives us that our computed $LU$ factorization obeys the relationship ${\bf A}  + E = LU$ such that   
\[ \left| E_{jk} \right| \leq \frac{n \epsilon_{mach}}{1-n\epsilon_{mach}} \left( \left| L \right| \cdot \left| U \right| \right)_{jk},  \]
where $E \in \mathbb{R}^{n\times n}$ and $\left| \cdot \right|$ denotes taking the absolute value of each entry of a matrix. For linear systems $LUx = b$, the backwards stability of forward/backward substitution tells us that our computer will calculate $\hat{x}$ which satisfies the following approximation
\[ \left( L + \delta L \right) \left( U + \delta U \right) \hat{x} = b \]
Thus, Gaussian elimination for the linear system is backwards stable if we can provide a tight bound on $\delta A$ such that $(A+\delta A) \hat{x} = b$ where
\[ \delta{\bf A}= \delta L U + L \delta U + \delta L \delta U + E \]
Theorem 9.4 of \cite{highambk} tells us that $\delta A$ must satisfy
\begin{equation} \label{deltaAform} \left| \delta{\bf A}\right| \leq \frac{3n\epsilon_{mach}}{1-3n\epsilon_{mach}} \left| L \right| \left| U \right|  \end{equation}
Therefore, in order to bound $\delta A$, we need to simply bound $L$ and $U$ of our computed factorization. Next, we define a property that some pivoting strategies enjoy.
\begin{definition} [Top-Heavy Pivoting Strategies] \label{topheavy}
A pivoting strategy for Gaussian elimination or the LU decomposition is called {\bf top-heavy} if the pivoting strategy leaves the first entry of the leading column of each Schur complement to be the entry with largest modulus in the leading column. In other words,
\begin{equation} \label{eqn:topheavy} \left| S_k^{\Pi_c}(k,k) \right| = \max_{k \leq i \leq n} \left| S_k^{\Pi_c}(i,k) \right| \end{equation}
or, in other words
\[ \left| S_k^{\Pi_c}(k,k) \right| = \| S_k^{\Pi_c}(:,k) \|_{\infty} \]
for all $1\leq k \leq n$.
\end{definition}
All of the methods discussed in this paper (partial, complete, rook, $\ell_2$ complete and randomized complete pivoting) are {\bf top-heavy} pivoting strategies. Therefore, all of these strategies enjoy the following property
\begin{lemma}
Let ${\bf A} \in\mathbb{R}^{n\times n}$ and let the lower triangular matrix $L$ be obtained by the LU algorithm above with a top-heavy pivoting strategy. Then, we have
\[ \left\| L \right\|_{p} \leq n \]
where $1\leq p\leq \infty$.
\end{lemma}
\begin{proof}
First, note that top-heaviness gives us that
\[ \left|l_{jk}\right| = \frac{\left| S_k(j,k) \right|}{\left| S_k(k,k) \right|} \leq 1 \] 
for $k \leq j$. Thus, 
\[ \|L\|_1 \leq n, \hspace{1cm} \|L\|_{\infty} \leq n \hspace{0.5cm} \text{and} \hspace{0.5cm} \|L\|_p \leq \| L \|_1^{\frac{1}{p}} \| L \|_{\infty}^{1-\frac{1}{p}} \leq n \]
where the last inequality was established by Theorem \ref{2_from_inf_and_1}. 
\end{proof} 

It is easy to see that all top-heavy strategies enjoy the bound $\left\| L \right\|_{p} \leq n$ for $1\leq p\leq \infty$ because all the entries below the diagonal in $L$ will be between $-1$ and $1$ in addition to Lemma \ref{2_from_inf_and_1}. Bounding $U$ is a little trickier and historically, the element growth factor has been used to do it.
\begin{definition} [Element Growth Factor] 
Let $n > 0$ and ${\bf A}  \in \mathbb{R}^{n\times n}$ 
\[\rho_{elem} \left({\bf A}\right) \stackrel{def}{=} \frac{\max_k \left\| S_k \right\|_{1,\infty}}{\left\|{\bf A}\right\|_{1,\infty}} = \frac{\max_{i,j,k} \left| S_k(i,j) \right|} {\max_{i,j} \left| {\bf A}(i,j) \right| }\]
\end{definition}
In addition to the classical element growth factor, we define the new column growth factor which will be central to our analysis.
\begin{definition} [Column Growth Factor] 
Let $n > 0$ and ${\bf A}  \in \mathbb{R}^{n\times n}$ 
\[\rho_{col} \left({\bf A}\right) \stackrel{def}{=} \frac{\max_k \left\| S_k \right\|_{1,2}}{\left\|{\bf A}\right\|_{1,2}} = \frac{\max_{j,k} \left\| S_k(:,j) \right\|_2} {\max_{j} \left\| {\bf A}(:,j) \right\|_2 }\]
\end{definition}
These two definitions of the growth factor are related by the following lemma. It is important to note that the column growth factor commonly attains the lower bound of $\frac{1}{\sqrt{n}} \rho_{elem}$ as we will see with partial and complete pivoting; making $\rho_{col}$ a more informative quantity to control than $\rho_{elem}$ by a factor of $\sqrt{n}$.
\begin{lemma} \label{lfactorbound}
Let $n>0$ and ${\bf A}  \in \mathbb{R}^{n\times n}$, then
\[ \frac{1}{\sqrt{n}} \rho_{elem} \left({\bf A}\right) \leq \rho_{col} \left({\bf A}\right) \leq \sqrt{n} \rho_{elem} \left({\bf A}\right) \]
\end{lemma}
\begin{proof} Easy consequence of Lemma \ref{inf_to_2}. \end{proof}

Using the definition of element growth, we can bound $U$ in the $\| \cdot \|_1$ and $\| \cdot \|_{\infty}$ norms as
\[ \left\| U \right\|_{\eta} \leq n \rho_{elem} \left({\bf A}\right) \left\|{\bf A}\right\|_{1,\infty} \hspace{1cm} \text{for } \eta= 1,\infty \]
Plugging both of these estimates into Theorem \ref{2_from_inf_and_1}, we get that
\[ \left\| U \right\|_{p} \leq n \rho_{elem} \left({\bf A}\right) \left\|{\bf A}\right\|_{1,\infty} \hspace{1cm} \text{for } 1\leq p\leq \infty \]
This gives us the following classical result  
\begin{theorem} [Wilkinson \cite{demmelbk,highambk}] \label{elemback} Let ${\bf A}  \in \mathbb{R}^{n\times n}$ and let $E$ and $\delta {\bf A}$ be given above, then for any top-heavy pivoting strategy, we have 
\begin{eqnarray*}
\left\| \delta{\bf A}\right\|_{p} &\leq& \frac{3 n\epsilon_{mach}}{1-3n\epsilon_{mach}} n^2 \rho_{elem} \left({\bf A}\right) \left\|{\bf A}\right\|_{1,\infty} \\
\left\| E \right\|_{\eta} &\leq& \frac{ n\epsilon_{mach}}{1-n\epsilon_{mach}} n^2 \rho_{elem} \left({\bf A}\right) \left\|{\bf A}\right\|_{1,\infty} 
\end{eqnarray*}
where $1\leq p\leq \infty$.
\end{theorem}
The proof of the above is similar to the proof of the corresponding result for the column growth factor.
\begin{theorem} [Column growth control on backward error] \label{colback}
Let ${\bf A}  \in \mathbb{R}^{n\times n}$ and let $E$ and $\delta {\bf A}$ be given above, then for any top-heavy pivoting strategy, we have 
\begin{eqnarray*}
\left\| \delta{\bf A}\right\|_{p} &\leq& \frac{3 n\epsilon_{mach}}{1-3n\epsilon_{mach}} n^2 \rho_{col} \left({\bf A}\right) \left\|{\bf A}\right\|_{1,2} \\
\left\| E \right\|_{p} &\leq& \frac{n\epsilon_{mach}}{1-n\epsilon_{mach}} n^2 \rho_{col} \left({\bf A}\right) \left\|{\bf A}\right\|_{1,2}
\end{eqnarray*}
where $1\leq p\leq \infty$.
\end{theorem}
\begin{proof}
First, Lemma \ref{lfactorbound} gives use the desired bound on $\|L\|_{p}$. Next, please note that by the definition of the LU algorithm, we have $U(m,m:n)=S_m(\beta_m,\pi_{(m,\alpha_m)}(m:n))$. Next, we tackle the $U$ matrix will $\rho_{col}$.
\begin{align*}
\left\| U \right\|_{\infty} &= \max_{1 \leq m \leq n} \left\| U(m,:) \right\|_1 
\leq \sqrt{n} \max_{1 \leq m \leq n} \left\| U(m,:) \right\|_2 
= \sqrt{n} \max_{1 \leq m \leq n} \left\|  S_m(\beta_m,:) \right\|_2 \\
&= \sqrt{n} \max_{1 \leq m \leq n} \left\| S_m \right\|_F 
\leq n \max_{1 \leq m \leq n} \max_{m \leq j \leq n} \left\| S_m(:,j) \right\|_2 
= n \max_{1 \leq m \leq n} \left\| S_m \right\|_{1,2}
\end{align*}
and 
\begin{align*}
\left\| U \right\|_{1} &= \max_{1 \leq j \leq n} \left\| U(:,j) \right\|_1 
\leq n \max_{1 \leq j \leq n} \left\| U(:,j) \right\|_{\infty} 
= n \max_{1 \leq j \leq n} \max_{1 \leq m \leq j} \left| U(m,j) \right| \\
&= n \max_{1 \leq m \leq n} \max_{m \leq j \leq n} \left| S_m(m,j) \right| 
\leq n \max_{1 \leq m \leq n} \max_{m \leq j \leq n} \left\| S_m(:,j) \right\|_2 
= n \max_{1 \leq m \leq n} \left\| S_m \right\|_{1,2}
\end{align*}
Thus, Theorem \ref{2_from_inf_and_1} gives us
\begin{align*}
\left\| U \right\|_p &\leq \left\| U \right\|_1^{\frac{1}{p}} \left\| U \right\|_{\infty}^{1-\frac{1}{p}} \leq n \max_{1 \leq m \leq n} \left\| S_m \right\|_{1,2}
\end{align*}
for all $1\leq p\leq \infty$. Note that
\[ \max_{1 \leq m \leq n} \left\| S_m \right\|_{1,2} = \frac{\max_{1 \leq m \leq n} \left\| S_m \right\|_{1,2}}{\left\|{\bf A}\right\|_{1,2}} \left\|{\bf A}\right\|_{1,2} = \rho_{col} \left({\bf A}\right) \left\|{\bf A}\right\|_{1,2} \]
We arrive at our conclusion by combining this with equation (\ref{deltaAform}).
\end{proof}
\subsection{Popular Pivoting Strategies}
\subsubsection{No Pivoting or Static Pivoting}
The easiest and most computationally efficient strategy is not to pivot at all, which we call \emph{Gaussian elimination with no pivoting} (GENP). However, this method is not numerically stable for a general ${\bf A}  \in \mathbb{R}^{n\times n}$. To demonstrate this, we use an example from \cite{highambk}
\[{\bf A}= \left[ \begin{array}{cc} \delta & -1 \\ 1 & 1 \end{array} \right] = \left[ \begin{array}{cc} 1 & 0 \\ \delta^{-1} & 1 \end{array} \right] \left[ \begin{array}{cc} \delta & -1 \\ 0 & 1+\delta^{-1} \end{array} \right] = L U \]
where $\delta < \epsilon_{mach}$ and $\epsilon_{mach}$ is machine epsilon. However, in floating point arithmetic, we have $fl \left(1+\delta^{-1} \right) = \delta^{-1}$ where $fl(\cdot)$ represents evaluation in floating point arithmetic. Thus,
\[fl\left( L U \right) = fl\left( L \right) fl\left( U \right) =\left[ \begin{array}{cc} 1 & 0 \\ \delta^{-1} & 1 \end{array} \right] \left[ \begin{array}{cc} \delta & -1 \\ 0 & \delta^{-1} \end{array} \right] = \left[ \begin{array}{cc} \delta & -1 \\ 1 & 0 \end{array} \right] \neq{\bf A}\]
which is the wrong computed LU factorization with backward error $\left\|{\bf A}- fl(L) fl(U) \right\|_{\infty} = 1$. On the other hand, if ${\bf A}  \in \mathbb{R}^{n \times n}$ is
\begin{enumerate}
\item[$\bullet$] Totally nonnegative, i.e. the determinant of every square submatrix is nonnegative,
\item[$\bullet$] Row or column diagonally dominant,
\item[$\bullet$] Symmetric positive definite,
\end{enumerate} 
then it is proved in \cite{highambk} that \emph{no pivoting} is required for a stable computation. In fact, all of these matrices have element growth $\rho_n = O(1)$. 
\subsubsection{Partial Pivoting} \label{sec:gepp}
The most common version of LU is \emph{Gaussian elimination with partial pivoting} (GEPP) because it provides some stability at relatively cheap overhead. It is backward stable ``in practice'' \cite{demmelbk}, meaning that this method provides a stable LU factorization for most but not all matrices. For partial pivoting, please place the following pivoting rule in Algorithm $1$: At the $k^{th}$ stage, the $k^{th}$ row is swapped with the $\beta_k^{th}$ row, where
\[\beta_k = \argmax_{k \leq i \leq n} \left| S_k(i,k) \right| \]
This method requires a number of entry comparisons in addition to the floating point operations required by Gaussian elimination without pivoting. Specifically, it requires 
\[ \sum_{k=1}^n (n-k) = \frac{n(n-1)}{2} \]
comparisons in total, which is one order less asymptotically than the $\frac{1}{3}n^3 + O(n^2)$ flops in GENP. Each row swap requires $n$ individual entry swaps, so the total number of entry swaps required is bounded above by
\[ \sum_{k=1}^n n = n^2\]
Since data movement is typically far more expensive than comparisons per operation, the dominant overhead cost is typically the data movement involved with swapping. The element growth for partial pivoting is bounded by
\[ \rho_{elem}^{gepp} \left({\bf A}\right) \leq  2^{n-1} \]
It is also easy to show that the column growth for partial pivoting is bounded by
\[ \rho_{col}^{gepp} \left({\bf A}\right) \leq \frac{1}{\sqrt{n}} 2^{n-1} \]
Both of these bounds are attained by the Wilkinson matrix
\[ \left[ \begin{array} {cccccc} 1 & 0 & 0 & \cdots & 0 & 1 \\ -1 & 1 & 0 & \cdots & 0 & 1 \\ -1 & -1 & 1 & \cdots & 0 & 1 \\ \vdots & \vdots & \vdots & \ddots & \vdots & \vdots \\ -1 & -1 & -1 & \cdots & 1 & 1 \\ -1 & -1 & -1 & \cdots & -1 & 1 \end{array} \right] = \left[ \begin{array} {cccccc} 1 & 0 & 0 & \cdots & 0 & 0 \\ -1 & 1 & 0 & \cdots & 0 & 0 \\ -1 & -1 & 1 & \cdots & 0 & 0 \\ \vdots & \vdots & \vdots & \ddots & \vdots & \vdots \\ -1 & -1 & -1 & \cdots & 1 & 0 \\ -1 & -1 & -1 & \cdots & -1 & 1 \end{array} \right] \left[ \begin{array} {cccccc} 1 & 0 & 0 & \cdots & 0 & 1 \\ 0 & 1 & 0 & \cdots & 0 & 2 \\ 0 & 0 & 1 & \cdots & 0 & 2^2 \\ \vdots & \vdots & \vdots & \ddots & \vdots & \vdots \\ 0 & 0 & 0 & \cdots & 1 & 2^{n-2} \\ 0 & 0 & 0 & \cdots & 0 & 2^{n-1} \end{array} \right] = LU \]
where we see that the element growth in $U$ is $2^{n-1}$ and the column growth $\frac{\| 2^{n-1} \|_{\ell_2\left( \mathbb{R}^1 \right)}}{\| {\bf e} \|_{\ell_2\left( \mathbb{R}^n \right)}} = \frac{2^{n-1}}{\sqrt{n}}$. 

Large element growth and unstable LU factorizations with partial pivoting also occur in many applications. Wright \cite{wrightfail} describes a family of two-point boundary value problems that cause GEPP to fail via exponential element growth when attempting to solve the ODE by discretizing it into matrix form. Liu and Russell \cite{liu1993linear} experience the same phenomenon when attempting to solve the discretized Kuramoto-Sivashinsky PDE, which is used to model laminar flame front propagation, phase dynamics in reaction-diffusion systems, fluctuations in fluid films and instabilities in plasma physics \cite{HNks86, laqueynonlinear}. Foster \cite{FOSTER4} applies the Newton-Cotes quadrature  to discretize the Volterra Integral equation from many areas of applied mathematics including actuarial science, viscoelastic materials and probability theory. This reduces the Volterra Integral equation into a matrix equation that makes GEPP fail. As we will discuss later, the Volterra example is among the most diabolical examples that break GEPP because it induces \emph{passive aggressive element growth}, i.e. barely enough element growth to cause GEPP to fail. 

The remainder of the partial pivoting section is spent discussing the \emph{generalized Wilkinson Matrix}, which is a more general class of matrices that can cause exponential element growth in GEPP.  
\begin{EXAMPLE}[Generalized Wilkinson Matrices $\mathcal{GW}$] For any integer $r \geq 1$, consider a matrix ${\bf A}$ of the following form
\[{\bf A}= L + \left(\begin{array}{c}
1  \\
\vdots \cr
1 \cr
0 \end{array}\right) \left(\begin{array}{cccc}
0  &
\dots & 
0 &
1 \end{array}\right) , \]
where $L$ is a lower triangular  matrix with 
\[L_{i,i} = 1, \quad \mbox{and} \quad L_{i,j} = - u_i^T W_{i+1} \cdots W_{j-1} v_j, \quad \mbox{for any} \quad i > j, \]
with $u_i, v_j \in \mathbb{R}^r$ being any vectors and $W_i \in \mathbb{R}^{r \times r}$ being square matrices.
The matrix ${\bf A}$ reduces to the Wilkinson matrix for the special case $r=1$, $u_i=v_j=W_i = 1$ for all $i$ and $j$. It is straightforward to verify that $L^{-1}$ is a lower triangular  matrix with 
\[\left(L^{-1}\right)_{i,i} = 1, \quad \mbox{and} \quad \left(L^{-1}\right)_{i,j} = u_i^T \widehat{W}_{i+1} \cdots \widehat{W}_{j-1} v_j, \quad \mbox{for any} \quad i > j, \]
where $\widehat{W}_{i} = W_i + v_i u_i^T$. Now we choose the vectors $\{u_i\}_{i=2}^n$, $\{v_j\}_{j=1}^{n-1}$ and matrices $\{W_i\}_{i=2}^{n-1}$ to contain only positive entries and have 2-norm at most $1$.  This implies that 
\[ |L_{i,j}| = |u_i^T W_{i+1} \cdots W_{j-1} v_j| \leq 1 \quad \mbox{for any} \quad i > j. \]

Consequently LU-factorizing ${\bf A}$ with GEPP will incur no row exchanges, and the resulting matrix factorization has the form
\[{\bf A}= L \; U, \quad \mbox{where} \quad U = I + \left(L^{-1} \left(\begin{array}{c}
1  \\
\vdots \cr
1 \cr
0 \end{array}\right)\right)  \left(\begin{array}{cccc}
0  &
\cdots & 
0 &
1 \end{array}\right). \]
This typically implies exponential element growth in $U$ if the inequality $\|\widehat{W}_{i}\|_2 > 1$ holds for most matrices $\widehat{W}_{i}$.
\end{EXAMPLE}

In our numerical experiments, we use this to create a random matrix ensemble that causes GEPP to fail with high probability.

\subsubsection{Complete Pivoting}
The most reliable version is \emph{Gaussian elimination with complete pivoting} (GECP). Von Neumann and Goldstine \cite{von1947numerical} referred to this as the ``customary procedure,'' making it arguably the first published pivoting strategy for Gaussian elimination in large matrix computations.
For complete pivoting, please place the following pivoting rule in Algorithm $1$: At $k^{th}$ stage, $k^{th}$ row and $k^{th}$ column are swapped with $\beta_k^{th}$ row and $\alpha_k^{th}$ column respectively , where
\[ \left( \beta_k, \alpha_k \right) = \argmax_{\substack{(i,j) \in \mathbb{N}^2 \\ k \leq i,j \leq n}} \left| S_k(i,j) \right| \]
Next, we consider the overhead of complete pivoting, which is broken into entry comparisons and data movement. The number of overall entry comparisons is 
\[ \sum_{k=1}^{n} \left(k^2 - 1 \right) = \frac{n(n+1)(2n+1)}{6} - n  = \frac{1}{3}n^3 + \frac{1}{2}n^2 - \frac{5}{6}n\]
In the worst case, the number of entry swaps (or data movement) is 
\[ \sum_{k=1}^n 2n = 2n^2 \]
Therefore, despite the fact that each entry swap is more expensive on modern computers than each entry comparison, the entry comparisons will form the bulk of the overhead when $n$ is large. This is different than partial pivoting because the total number of comparisons in complete pivoting is $O(n^3)$ instead of the $O(n^2)$ comparisons in partial pivoting. In his seminal work, Wilkinson \cite{wilk1} proves that the element growth in complete pivoting is bounded above by
\[ \rho_{elem}^{gecp}({\bf A}) \leq \sqrt{n} \left(2 \cdot 3^{\frac{1}{2}} \cdots n^{\frac{1}{n-1}} \right)^{1/2} \sim c n^{1/2} n^{\frac{1}{4} \ln(n)} \] 
Our proof of Theorem \ref{gercpgf} can be easily adapted to show that the column growth for complete pivoting is bounded above by 
\[ \rho_{col}^{gecp}(A) \leq \left(2 \cdot 3^{\frac{1}{2}} \cdots n^{\frac{1}{n-1}} \right)^{1/2} \sim c n^{\frac{1}{4} \ln(n)} \]
These bounds are provably unattainable for $n \geq 3$ because of their proof's reliance on Hadamard's inequality of Theorem \ref{hadamard}. It was incorrectly conjectured that $\rho_{elem}^{gecp}({\bf A}) \leq n $ \cite{edelman1992complete,gould1991growth}. Nonetheless, it is widely believed that the above element growth bound is wildly pessimistic. However, this bound proves that exponential element growth is impossible as $n^{\log(n)}$ is sub-exponential. Because of this, we call GECP \emph{backwards stable}.  
\subsubsection{Rook Pivoting} \label{subsecrook}
An important attempt to speed up complete pivoting was introduced by Neal and Poole \cite{geogauss2} as \emph{Gaussian elimination with rook pivoting} (GERP). Basically, one alternates between partial pivoting on the rows and the columns until arriving at a Nash-equilibrium of sorts. For rook pivoting, place the following pivoting rule in Algorithm $1$: At $k^{th}$ stage, initialize $\beta_k=k$ and $\alpha_k=k$. First, choose a new $\beta_k$ from (\ref{rookstage1}) while holding $\alpha_k$ constant, and then choose $\alpha_k$ from (\ref{rookstage2}) while holding $\beta_k$ constant. Repeat (\ref{rookstage1}) and (\ref{rookstage2}) until the current choice $(\beta_k,\alpha_k)$ make (\ref{rookstage1}) and (\ref{rookstage2}) hold simultaneously.
\begin{align}
\beta_k &= \argmax_{k \leq i \leq n} \left| S_k \left( i, \alpha_k \right) \right| \label{rookstage1} \\
\alpha_k &= \argmax_{k \leq j \leq n} \left| S_k \left( \beta_k, j \right) \right| \label{rookstage2} 
\end{align}
The number of entry comparisons required by rook pivoting depends on the matrix. For example, a diagonally dominant matrix will require no swaps, which causes rook pivoting to only check (\ref{rookstage1}) once. This example gives $\sum_{k=1}^n (n-k) = \frac{n(n-1)}{2} = O(n^2)$ entry comparisons just like partial pivoting. In fact, there are a few probabilistic arguments \cite{foster3,rookpoole} that claim for the ``average'' input matrix ${\bf A}$, the user would expect to perform $O(n^2)$ entry comparisons in total. However, the worst case needs a total of $O(n^3)$ entry comparisons (just as bad as complete pivoting) as exemplified by any matrix of the form \cite{highambk}
\begin{align*}
\left( \begin{array}{ccccc} \theta_1 & \theta_2 & & & \\ & \theta_3 & \theta_4 & & \\ & & \ddots & \ddots & \\ & & & \theta_{2n-3} & \theta_{2n-2} \\ & & & & \theta_{2n-1} \end{array} \right), \hspace{2cm} \left| \theta_1 \right| < \left| \theta_2 \right| < \cdots < \left| \theta_{2n-1} \right|
\end{align*}
The worst case data movement in terms of entry-wise swaps is the same as GECP at $O(n^2)$ because both rows and columns are also being swapped here. It is not clear whether comparisons or swaps are to be considered the dominant overhead cost as it will depend on whether the input matrix requires $O(n^3)$ or $O(n^2)$ comparisons. Foster \cite{foster3} proves that rook pivoting element growth must obey 
\[ \rho^{gerp}_{elem}({\bf A}) \leq \frac{3}{2} n^{\frac{3}{4} \ln(n)} \] 
and he also shows that the bound is unattainable for $n\geq 3$. This suggests that this method has similar stability properties to complete pivoting and should be considered as a less expensive ``cousin.''
\subsubsection{Prior attempts at randomizing Gaussian elimination} \label{subsecprirand}
Randomization in the context of Gaussian elimination based direct solvers has been attempted in the past as a way to avoid pivoting altogether \cite{pan2013randomized, parker1995random}. These methods pre/post multiply our input matrix ${\bf A}$ by random matrices before applying Gaussian elimination without pivoting. This can either be used to solve a linear system or compute the inverse of ${\bf A}$. Unfortunately, there are no theoretical guarantees of small backwards error. In fact, many examples, including the Turing matrix, cause these methods to out-right fail in terms of backwards error as shown in the numerical experiments section of \cite{parker1995random}. 

\section{Algorithm GERCP and Main Results}\label{Sec:GERCP}

In this section, we first introduce Algorithm \ref{Alg:ge2cp}, a deterministic complete pivoting scheme based on the $\ell_2$ norm;  we then evolve Algorithm \ref{Alg:ge2cp} into GERCP by significantly speeding it up with randomization.

\subsection{A Deterministic $\ell_2$-norm Complete Pivoting Strategy} \label{subsectl2piv}
To help motivate our randomized strategy, we propose an intermediate deterministic strategy called \emph{Gaussian elimination with $\ell_2$ complete pivoting} (GE2CP).

\begin{algorithm}\label{Alg:ge2cp} {\bf Gaussian Elimination with $\ell_2$ norm Complete Pivoting (GE2CP)} \\
\headerrule

\begin{tabular}{ll}
{\bf Input:} & $n \times n$ matrix ${\bf A}$ \\
{\bf Output:} & lower triangular $L$, upper triangular $U$, row permutation $\Pi_r$, column permutation $\Pi_c$.  \\
\end{tabular}

\noindent\headerrule

\begin{INDENT}
{\noindent \bf for} $k = 1, \cdots, n - 1$ {\bf do} \\
\begin{INDENT}
\begin{enumerate}
\item {\bf compute}  ${\displaystyle \alpha = {\bf argmax}_{k \leq j \leq n} \left\| A(k:n,j) \right\|_2}$. \\
 {\bf swap} columns $k$ and $\alpha$ of $A$.
\item {\bf compute} ${\displaystyle \beta = {\bf argmax}_{k \leq j \leq n} \left| A(j,k) \right|}$. \\
{\bf swap} rows $k$ and $\beta$ of $A$.
\item {\bf compute} $A(k+1:n,k) = A(k+1:n,k)/A(k,k)$;
\item {\bf compute} $A(k+1:n,k+1:n) = A(k+1:n,k+1:n) - A(k+1:n,k) * A(k,k+1:n) $;
\end{enumerate} \\
\end{INDENT}
\end{INDENT}
\headerrule
\end{algorithm}

Pivoting is done in two steps in Algorithm \ref{Alg:ge2cp} for each $k$: Step 1 swaps the $\ell^{th}$ column of the trailing matrix $A(k:n,k:n)$ with the $k^{th}$, where $A(k:n,\alpha)$ has the largest column $2$-norm among all columns of $A(k:n,k:n)$; whereas Step 2 swaps the $\beta^{th}$ row of $A(k:n,k:n)$ with the $k^{th}$, where $A(\beta,k)$ is the largest in absolute value among all entries of $A(k:n,k)$. Step 1 controls potentially harmful column norm growth in ${\bf A}$ through column interchanges, and Step 2 performs standard partial pivoting to control potentially harmful element growth in the $k$-column $A(k:n,k)$ through row interchanges. Step 2 makes the $\ell_2$ norm Complete Pivoting strategy (GE2CP) in Algorithm \ref{Alg:ge2cp} a \emph{top-heavy} pivoting strategy, which allows us to apply Theorems \ref{colback} to bound the LU backward error.

In this method, the number of comparisons is reduced in favor of additional floating point operations. The total amount of comparisons for this method is
\[ \sum_{k=1}^n 2(n-k) = n(n-1) = O(n^2) .\]
The total additional floating point operations required to directly compute the $2$-norms in Step 1 of Algorithm \ref{Alg:ge2cp} is about 
\[ \sum_{k=1}^{n-1} 2(n-k+1)(n-k+1) = \frac{n(n+1)(2n+1)}{3} -2 = O(n^3).\]
The worst case for entry swaps in GE2CP is the same as in GECP at $O(n^2)$. The $\ell_2$ norm complete pivoting strategy (GE2CP) in Algorithm \ref{Alg:ge2cp} enjoys similar element/column growth bounds to complete pivoting and rook pivoting. By omitting references and applications of Lemma \ref{Lem:WellPreservedColNrm} from the proof of Theorem \ref{gercpgf}, one can easily show that
\[ \rho^{ge2cp}_{col} \leq n \sqrt{e(n+1)} n^{\frac{1}{2} \ln \left( n \right)} \]
The $O(n^3)$ flop overhead makes Algorithm \ref{Alg:ge2cp} an impractical alternative to GEPP. In Section \ref{SubSec:GERCP}, we develop GERCP by randomizing Algorithm \ref{Alg:ge2cp} to choose columns with sufficiently large column norms at significantly lowered overhead costs. Furthermore, we will derive a GECP style element growth upper bound for GERCP that holds with arbitrarily high probability.

\subsection{Gaussian Elimination with Randomized Complete Pivoting (GERCP)}\label{SubSec:GERCP} 
Our randomized Gaussian Elimination algorithm, GERCP, is based, in principle, on  Algorithm \ref{Alg:ge2cp}. However, the key difference is we will replace the column pivoting step, Step 1, by a randomized alternative to significantly reduce its cost.  At its core, GERCP relies on a fast random project scheme to reliably estimate the column norms of {\em each} observed Schur complement required by line $1$ of Algorithm \ref{Alg:ge2cp}, based on Theorem \ref{randproj},

\subsubsection{Successive Schur Sketching} 
We adapt the idea of a \emph{sketching matrix} \cite{woodruff2014sketching} in this section to speed up the column selection procedure in GE2CP Algorithm \ref{Alg:ge2cp}. Let our sampling matrix $\Omega \in \mathbb{R}^{r \times n}$ be a random matrix with iid standard normal $\mathcal{N}(0,1)$ entries where $r \ll n$, and let $\Psi = \Omega A$ for our first sketching matrix. We refer to the number $r \in \mathbb{N}$ as the {\bf sampling dimension}. One can also make use of a Fast Johnson-Lindenstrauss-like sampling matrix \ref{ailon2009fast} as a sampling matrix. However, we will stick with a Gaussian sampling matrix in this paper for simplicity of presentation. By Theorem \ref{randproj}, the column norms of $A$ can be reliably estimated via those of $\Psi$ for a large enough choice of $r$. In other words, we can perform Step 1 of Algorithm \ref{Alg:ge2cp} for $k = 1$ by looking for the column with the largest column norm in $\Psi$. 

However, Step 1 of Algorithm \ref{Alg:ge2cp} must be performed for every value of $k$. From the single random matrix $\Omega$, below we will construct a collection of matrices $\left\{ \Psi_k \right\}_{k=1}^n$ known as the {\bf Schur sketching matrices}, where $\Psi_k \in \mathbb{R}^{r\times (n-k+1)}$. These matrices are constructed as
\begin{equation}\label{Eqn:SSS}
\Psi_k \stackrel{def}{=} \Omega\left(:,\pi_{r,k}(k:n)\right) \; S_k \text{, for all $1< k \leq n$.}
\end{equation}

\begin{remark} \label{randomschur}
For randomized complete pivoting, the choice of the $k^{th}$ pivot column will be based on column norms in the Schur sketching matrix $\Psi_k$. Since all Schur sketching matrices are based on the same random matrix $\Omega$, the Schur complement $S_k \in \mathbb{R}^{(k:n)\times (k:n)}$ for $1 < k \leq n$ will {\bf not} be a deterministic matrix. Indeed, the observed Schur complement $S_k$ is determined by the randomized column pivoting decisions $\alpha_1, \cdots, \alpha_{k-1}$ of GERCP from the previous stages. Given that there are only a finite number of pivot decisions, we conclude that $S_k$ must be a {discrete random variable}.
\end{remark}

To efficiently continue with all other column pivoting work in Algorithm \ref{Alg:ge2cp}, we inductively devise a scheme to use our current Schur sketching matrix $\Psi$ to produce our next Schur sketching matrix $\widehat{\Psi}$. Suppose that we have chosen $p>0$ rows/columns from the remaining $m$ rows/columns to pivot to the front of our working matrix
\[ {\displaystyle \left(\begin{array}{cc}
A_{11} &  A_{12}\\
 A_{21} &  A_{22} 
\end{array} \right) = 
\left(\begin{array}{cc}
L_{11} & \\
 L_{21} &  I  
\end{array} \right) \left(\begin{array}{cc}
U_{11} &  U_{12}\\
  &  \widehat{A}_{22} 
\end{array} \right) , } \]
where $A_{11}, L_{11}, U_{11} \in \mathbb{R}^{p\times p}$, with $L_{11}$ lower triangular and $U_{11}$ upper triangular, respectively; $A_{21}, L_{21} \in \mathbb{R}^{(m-p)\times p}$; $A_{12}, U_{12} \in \mathbb{R}^{p \times (m-p)}$; and $A_{22}, \widehat{A}_{22} \in \mathbb{R}^{(m-p)\times (m-p)}$ with 
\[ \widehat{A}_{22} = A_{22} - A_{21} A_{11}^{-1} A_{12} = A_{22} - L_{21} U_{12} \]
being the Schur complement. We call $p>0$ the {\bf pivot-block size}, namely the number of row/column pivots performed before each update to the Schur sketching matrix $\Psi$.  With the notation established above, we have $\widehat{A}_{22}=S_{n-m+p+1}$.

Step 1 of Algorithm \ref{Alg:ge2cp} requires that we perform column pivoting on $\widehat{A}_{22}$. To do this work, we need to multiply $\widehat{A}_{22}$ by a random matrix. Instead of generating a new random matrix, we introduce a simple and efficient procedure, {\bf Successive Schur Sketching (SSS)}. We partition $\Omega$ and $\Psi$ accordingly as ${\displaystyle  \Omega = \left(\begin{array}{cc}
\Omega_{P} &  \Omega_{R}
\end{array} \right)}$ and 
\begin{eqnarray*}
{\displaystyle \Psi} &= & {\displaystyle \left(\begin{array}{cc}
\Psi_{P} &  \Psi_{R}
\end{array} \right) = \left(\begin{array}{cc}
\Omega_{P} &  \Omega_{R}
\end{array} \right) \left(\begin{array}{cc}
A_{11} &  A_{12}\\
 A_{21} &  A_{22} 
\end{array} \right) } \\
& = &  {\displaystyle \left(\begin{array}{cc} 
\Omega_{P} A_{11}+ \Omega_{R} A_{21} & \Omega_{P} A_{12}+ \Omega_{R} A_{22}\end{array} \right). } 
\end{eqnarray*}

Now we compute the next Schur sketching matrix $\widehat{\Psi}$ for the Schur complement $\widehat{A}_{22}$ as
\begin{eqnarray}
{\displaystyle \widehat{\Psi}} &\stackrel{def}{=} &{\displaystyle \Omega_{R} \widehat{A}_{22}
= \Omega_{R}\left(A_{22} - L_{21} U_{12} \right)
 = \left(\Omega_{P} A_{12}+ \Omega_{R} A_{22}\right) - \left(\Omega_{P} A_{12}+ \Omega_{R} L_{21} U_{12} \right) } \nonumber\\
 & = &{\displaystyle \Psi_{R} - \left(\Omega_{P} L_{11}+ \Omega_{R} L_{21}\right) U_{12} . }\label{Eqn:PsiUpdate}
\end{eqnarray}

Thus, in SSS we use $\Omega_{R}$, a submatrix of $\Omega$, for the new sample matrix $\widehat{\Psi}$.
If done directly, it would take $2r(m-p)^2$ flops to compute $\widehat{\Psi}$ as a matrix-matrix product $\Omega_{R} \widehat{A}_{22}$. However, since $\Psi_{R}$ is part of $\Psi$ and was computed in the previous steps,  $\widehat{\Psi}$ can instead be computed via equation \eqref{Eqn:PsiUpdate} in about $4rmp$ flops, a very large savings for $p \ll m$. Indeed, this random matrix reuse will prove critically important in reducing the overall cost of computing all sample matrices by Algorithm GERCP. Later on we will further show that Algorithm GERCP will be as reliable as sampling the Schur matrices $\widehat{A}_{22}$ without random matrix reuse. 

Given $\Psi$, the new sample matrix $\widehat{\Psi}$ can be updated in about $4rpm$ flops with the above formula, which is much cheaper than the $O(rn^2)$ flops required for a direct computation. It costs another $2rm$ flops to select a column with sufficiently large column norm. If we perform a column pivot once every block elimination step, the total overhead due to column pivoting includes the computation of the initial sampling matrix, its update at every column pivot, and column selection based on the column norms of the updated sample matrix. These costs add up to 
\[ 4 r n^2 + \sum_{k=1}^{n/p} \left(4rp + 2r\right) \left(n-(k-1)p\right) \approx \left(6r + r/p\right) n^2 \]
flops, which is much smaller than the $O(n^3)$ additional flops required by GE2CP. One can preform a floating point error analysis to show that successively updating $\Psi$ at each stage via equation \eqref{Eqn:PsiUpdate} as follows.
\begin{lemma} Suppose equation \eqref{Eqn:PsiUpdate} is continuously used to produce each sketching matrix $\Psi_k$. Then, each sketching matrix in floating point arithmetic is close to the corresponding sketching matrix in real arithmetic as given by
\begin{align*}
\big| fl&\big( \widehat{\Psi}_{k+1} \big) - \widehat{\Psi}_{k+1} \big| \\ 
&\leq \epsilon_{mach} \left(1+ \epsilon_{mach} \right)^k \left(\sum_{i=1}^k \left| \Psi_{i} \left( :, (k+1)p:n \right) \right| + 5 \left| \Omega \right| \left| L_{(k+1)p} \right| \left| U_{(k+1)p}(:,(k+1)p:n) \right|\right) 
\end{align*}
\end{lemma}
\begin{remark}
Later in the paper, we will provide a probabilistic analysis of this lemma to get rid of the randomness originating from the sampling matrix $\Omega$.
\end{remark}
\begin{proof}
\noindent A brief floating point analysis proceeds as follows
\begin{align*}
\left| fl \big( (\Omega_{P} L_{11}+ \Omega_{R} L_{21})U_{12} \big) - (\Omega_{P} L_{11}+ \Omega_{R} L_{21})U_{12} \right| &\leq 4 \epsilon_{mach} (\left|\Omega_{P}\right| \left| L_{11} \right| + \left| \Omega_{R}\right| \left| L_{21}\right|)\left|U_{12}\right| \\
&=4 \epsilon_{mach} \left( \begin{array} {cc} \left| \Omega_P \right| & \left| \Omega_R \right| \end{array} \right) \left( \begin{array} {c} \left| L_{11} \right| \\ \left| L_{21} \right| \end{array} \right) \left| U_{12} \right| 
\end{align*}
Place this together with equation \eqref{Eqn:PsiUpdate} in floating point arithmetic to arrive at
\begin{align*}
\big| fl&\big( \widehat{\Psi}_{k+1} \big) - \widehat{\Psi}_{k+1} \big| \\ 
&\leq \underbrace{\left| fl\big( \Psi_{k,R} \big) - \Psi_{k,R} \right|}_{\text{past update errors}}\left(1+ \epsilon_{mach} \right) + \epsilon_{mach} \left| \Psi_{k,R} \right| + 5 \epsilon_{mach} \left( \begin{array} {cc} \left| \Omega_P \right| & \left| \Omega_R \right| \end{array} \right) \left( \begin{array} {c} \left| L_{11} \right| \\ \left| L_{21} \right| \end{array} \right) \left| U_{12} \right| 
\end{align*}
Applying simple induction argument to the inequality above gives the desired result.
\end{proof}
\begin{remark} This rounding error analysis applies equally well to the $L$ and $U$ factors computed in finite precision arithmetic. We did not make the distinction to avoid introducing yet more notation.
\end{remark}
When $U_{11}$ is well-conditioned, $\widehat{\Psi}$ can be updated by the more efficient formula
\begin{align} 
{\displaystyle \widehat{\Psi}} &\stackrel{def}{=} {\displaystyle \Omega_{R} \widehat{A}_{22}
= \Omega_{R}\left(A_{22} - L_{21} U_{12} \right)
 = \left(\Omega_{P} A_{12}+ \Omega_{R} A_{22}\right) - \left(\Omega_{P} A_{11}+ \Omega_{R} A_{21}\right) A_{11}^{-1}A_{12}  } \nonumber \\
 & = {\displaystyle \Psi_{R} - \left(\Psi_{P}U_{11}^{-1}\right)U_{12} , } \label{Eqn:PsiUpdateFast}
\end{align}
which costs about $2rpm + 2rm$ flops. 

\subsubsection{Column pivot quality and column growth factor for GERCP}
We present classical Gaussian elimination with randomized complete pivoting (GERCP) below as Algorithm \ref{Alg:gercp}. Provided that the sampling dimension $r$ is large enough, in this section, we show that, with high probability, each GERCP pivot column has an $\ell_2$ length within a constant factor of the largest column, i.e. the GE2CP pivot column. Using this property, we then prove a sub exponential upper bound on the column growth factor for GERCP that holds with probability not less than $1-\delta$ for any user-defined quantity $\delta >0$. For the sake of simplicity and theoretical guarantees, we will focus on the case with pivot-block size $p=1$ for the rest of the paper.
\begin{algorithm}\label{Alg:gercp}{\bf Gaussian Elimination with Randomized Complete Pivoting} \\
\headerrule

\begin{tabular}{ll}
{\bf Input:} & $n \times n$ matrix $A$, sampling dimension $r>0$ \\
{\bf Output:} & lower triangular $L$, upper triangular $U$, row permutation $\Pi_r$, column permutation $\Pi_c$.  \\
\end{tabular}

\noindent\headerrule \\
{\bf sample} $\Omega(i,j) \sim \mathcal{N}(0,1)$ for all $1\leq i\leq r$ and $1\leq j\leq n$ \\
{\bf compute} $\Psi = \Omega A$ \\
{\noindent \bf for} $k = 1, \cdots, n - 1$ {\bf do}
\begin{INDENT}
\begin{enumerate}
\item {\bf compute} ${\displaystyle \alpha = \argmax_{k \leq j \leq n} \left\| \Psi(:,j) \right\|_2}$. \\
 {\bf swap} columns $k$ and $\alpha$ of $A$ and $\Psi$.
\item {\bf compute} ${\displaystyle \beta = \argmax_{k \leq i \leq n} \left| A(i,k) \right|}$. \\
{\bf swap} rows $k$ and $\beta$ of $A$.
\item {\bf compute} $A(k+1:n,k) = A(k+1:n,k)/A(k,k)$;
\item {\bf compute} $A(k+1:n,k+1:n) = A(k+1:n,k+1:n) - A(k+1:n,k) * A(k,k+1:n) $;
\item {\bf update} $\Psi(:,k:n)$ with either \eqref{Eqn:PsiUpdate} or \eqref{Eqn:PsiUpdateFast} in real arithmetic
\end{enumerate} 
\end{INDENT}
\headerrule
\end{algorithm}

For ease of notation, we consider a scalar version of GERCP where column pivoting and row pivoting are done one column/one row at a time. The column pivots $\Pi_c$ are chosen from randomized column pivoting. The row pivots $\Pi_r$ are also random, but if we fix the column pivots $\Pi_c$ then the row pivots are deterministic and are uniquely determined by the top-heavy property (\ref{eqn:topheavy}). Let $S^{\Pi_c}_k \in \mathbb{R}^{(k:n)\times (k:n)}$ be the Schur complement after the first $k-1$ steps of column and row pivoting and Gaussian elimination, and let $\Omega^{\Pi_c}_k = \Omega(:, \pi_r(k:n)) \in \mathbb{R}^{r\times (n-k+1)}$ be the submatrix of $\Omega$ whose columns correspond to the rows of $S^{\Pi_c}_k$. We now define the following events for $1 \leq k \leq i \leq n$:
\begin{align}
\overline{{\bf C}}_{i,k}^{\Pi_c}  &=  {\displaystyle\left\{ \left\| \frac{1}{\sqrt{r}} \Omega^{\Pi_c}_k S^{\Pi_c}_k(:,i) \right\|_2 \leq  \sqrt{1+\epsilon} \left\| S^{\Pi_c}_k(:,i) \right\|_2 \right\}} \label{eqn:Cupper}\\
\underline{{\bf C}}_{i,k}^{\Pi_c}  &=  {\displaystyle\left\{ \sqrt{1-\epsilon}\left\| S^{\Pi_c}_k(:,i) \right\|_2 \leq \left\| \frac{1}{\sqrt{r}} \Omega^{\Pi_c}_k S^{\Pi_c}_k(:,i) \right\|_2 \right\}} \\
{\bf C}_{i,k}^{\Pi_c}  &=  {\displaystyle\left\{ \sqrt{1-\epsilon}\left\| S^{\Pi_c}_k(:,i) \right\|_2 \leq \left\| \frac{1}{\sqrt{r}} \Omega^{\Pi_c}_k S^{\Pi_c}_k(:,i) \right\|_2 \leq  \sqrt{1+\epsilon} \left\| S^{\Pi_c}_k(:,i) \right\|_2 \right\}} \\
&= \overline{{\bf C}}_{i,k}^{\Pi_c} \bigcap \underline{{\bf C}}_{i,k}^{\Pi_c}
\end{align}
 
By Definition \ref{Def:JLvec}, ${\bf C}_{i,k}^{\Pi_c}$ describes the event where the $i^{th}$ column of the $k^{th}$ Schur complement $S^{\Pi_c}_k$ satisfies the $\epsilon-$JL condition under random mapping $\Omega^{\Pi_c}_k $. We also define for each $k$ 
 \[ {\displaystyle  {\bf D}_{k,k}^{\Pi_c}}  =  {\displaystyle\left\{  \left\|\Omega^{\Pi_c}_k S^{\Pi_c}_k(:,k) \right\|_2 \geq g \left\| \Omega^{\Pi_c}_k S^{\Pi_c}_k \right\|_{1,2} \right\}. } \] 
which describes the event where no column is swapped under threshold randomized column pivoting at step $k$ for some fixed threshold $0 < g \geq 1$. We further define for $k < i$ and ${\displaystyle \gamma_k \stackrel{def}{=} \left\| \Omega^{\Pi_c}_k  S^{\Pi_c}_k \right\|_{1,2}}$,
\[ {\displaystyle  {\bf D}_{i,k}^{\Pi_c}}  =  {\displaystyle\left\{  \left\|\Omega^{\Pi_c}_k  S^{\Pi_c}_k(:,k) \right\|_2 < g \gamma_k, \; 
  \left\|\Omega^{\Pi_c}_k S^{\Pi_c}_k(:,j) \right\|_2 < \gamma_k, \; k < j \leq i, \, \; \left\|\Omega^{\Pi_c}_k  S^{\Pi_c}_k(:,i) \right\|_2 = \gamma_k. \right\}} \]
\noindent Thus, ${\bf D}_{i,k}^{\Pi_c}$ is the event where columns $k$ and $i$ are swapped.  Thus the event 
\begin{subequations}\label{Eqn:Events}
\begin{equation}
{\displaystyle {\cal D}^{\left(\Pi_c\right)} \stackrel{def}{=}\bigcap_{1 \leq k \leq n}  {\bf D}_{ \alpha_k,k}^{\Pi_c}} 
\end{equation}
uniquely defines the column permutation $\Pi_c$ and by extension, row permutation $\Pi_r$, and the event 
\begin{equation}
{\displaystyle {\cal C}^{\left(\Pi_c\right)} \stackrel{def}{=}\bigcap_{1 \leq k \leq n} \left(\left( \bigcap_{\substack{ k\leq j\leq n \\ j \neq \alpha_k}} \underline{{\bf C}}_{j,k}^{\Pi_c} \right) \bigcap \overline{{\bf C}}_{\alpha_k,k}^{\Pi_c}\right)}  
\end{equation}
\noindent defines the Gaussian elimination process where every column in every Schur complement satisfies the $\epsilon-$JL condition for the column permutation $\Pi_c$; and the event 

\begin{equation}
{\displaystyle {\cal C} \stackrel{def}{=}\bigcup_{\Pi_c}  \left( {\cal C}^{\left(\Pi_c\right)} \bigcap {\cal D}\left(\Pi_c\right) \right) }
\end{equation}

\noindent defines set of Gaussian elimination processes where every column in each Schur complement produced by the algorithm satisfies the $\epsilon-$JL condition during the factorization process. Note that event ${\cal C}$ describes the randomized Gaussian elimination process, whereas event ${\cal C}^{\left(\Pi_c\right)}$ describes a particular incidence of ${\cal C}$ conditional on a particular column permutation $\Pi_c$.

\end{subequations}

Since GERCP performs partial pivoting at every step of elimination, it will successfully compute an LU factorization of any given matrix. Additionally, the events ${\displaystyle {\cal D}^{\left(\Pi_c\right)}}$ over the set of permutations are mutually exclusive by definition. In other words, 
\begin{equation}\label{Eqn:GERCPSucceed}
{\displaystyle   \mathbb{P} \left( \bigcup_{\Pi_c} {\cal D}^{\left(\Pi_c\right)} \right) = \sum_{\Pi_c}  \mathbb{P} \left( {\cal D}^{\left(\Pi_c\right)} \right) = 1,}
\end{equation}
where the set union is over all possible permutations $\Pi_c$. Lemma \ref{Lem:WellPreservedColNrm} below shows that with large probability, in GERCP every column in every Schur complement satisfies the $\epsilon-$JL condition during the elimination process regardless of which permutation $\Pi_c$ is chosen. 

\begin{lemma}[Randomized norm preservation] \label{Lem:WellPreservedColNrm}
Given $\epsilon, \delta \in (0,1)$ and $g \in (0,1]$. Let $\Psi$ be defined by equation \eqref{Eqn:SSS}. Choose $  {\displaystyle r > \frac{4}{\epsilon^2-\epsilon^3} \ln \left( \frac{n(n+1)}{2\delta} \right)}. $
Then we must have 
\begin{eqnarray} \label{frac_of_max}
\left\| S_k \right\|_{1,2} \leq \frac{1}{g} \sqrt{\frac{1+\epsilon}{1-\epsilon}} \left\| S_k(:,\alpha_k) \right\|_{2}
\end{eqnarray}
for all $1\leq k \leq n$ with probability no less than $1- \delta$.  
\end{lemma}
\begin{remark}
This lemma tells us that the factor $\sqrt{\frac{1-\epsilon}{1+\epsilon}}$ plays a similar role to the {\bf column thresholding factor} $g$ as the two factors appear multiplied together in the same part of the above inequality. Therefore, it is fruitful to view and label the term $\sqrt{\frac{1-\epsilon}{1+\epsilon}}$ from Johnson-Lindenstrauss as the {\bf artificial column thresholding factor}.
\end{remark}
\begin{proof} By equations \eqref{Eqn:Events}, the event 
\[ {\displaystyle  \bigcup_{\Pi_c} \left( {\cal C}^{\left(\Pi_c\right)} \bigcap {\cal D}\left(\Pi_c\right) \right)} \subseteq \left\{ \left\| S_k \right\|_{1,2} \leq \frac{1}{g} \sqrt{\frac{1+\epsilon}{1-\epsilon}} \left\| S_k(:,\alpha_k) \right\|_{2}, \text{ for all $1\leq k\leq n$} \right\} \]
defines a subset of the set of outcomes that satisfies our desired result. To show this take any fixed choice of column pivots $\Pi_c$. It is trivial to get $\left\| S_k (:,\alpha_k) \right\| \leq \frac{1}{g} \sqrt{\frac{1+\epsilon}{1-\epsilon}} \left\| S_k(:,\alpha_k) \right\|_{2}$ without any probability. If our outcome is in both ${\cal C}\left(\Pi_c\right)$ and ${\cal D}\left(\Pi_c\right)$, then for any $1\leq k \leq j \leq n$ and $j \neq \alpha_k$
\begin{align*}
\left(1-\epsilon \right) \left\| S_k(:,j) \right\|_2^2 &\leq \left\| \frac{1}{\sqrt{r}} \Psi_k (:,j) \right\|_2^2 \hspace{2.2cm} (\text{by event } \underline{{\bf C}}_{j,k}^{\Pi_c} )\\
&\leq \frac{1}{g^2} \left\| \frac{1}{\sqrt{r}} \Psi_k (:,\alpha_k)\right\|_2^2 \hspace{1.45cm} (\text{by event } \mathcal{D}(\Pi_c) )\\
&\leq \frac{1}{g^2} \left( 1 + \epsilon \right) \left\| S_k(:,\alpha_k) \right\|_2^2 \hspace{1cm} (\text{by event } \overline{{\bf C}}_{\alpha_k,k}^{\Pi_c} )
\end{align*}
Choosing $j = \argmax_{k\leq j\leq n} \left\| S_k(:,j) \right\|_2$ gives us line (\ref{frac_of_max}) and our desired set containment. Next, we must bound the probability of success from below. It follows from the definition of conditional probability that 
\begin{eqnarray}
{\displaystyle   \mathbb{P} \left( \bigcup_{\Pi_c} \left( {\cal C}\left(\Pi_c\right) \cap {\cal D} \left( \Pi_c \right) \right)  \right)} &= & {\displaystyle \sum_{\Pi_c}  \mathbb{P} \left(  {\cal C}\left(\Pi_c\right) \cap {\cal D} \left( \Pi_c \right) \right) }\nonumber \\
& = & {\displaystyle \sum_{\Pi_c}  \mathbb{P} \left( {\cal C}\left(\Pi_c\right) \left| {\cal D}\left(\Pi_c\right) \right. \right) \,  \mathbb{P} \left( {\cal D}\left(\Pi_c\right) \right)} \label{Eqn:temp1}
\end{eqnarray}

Below we derive a lower bound on the right hand side of equation \eqref{Eqn:temp1}. Consider any given event ${\cal D}\left(\Pi_c\right)$ for which ${\displaystyle \mathbb{P} \left( {\cal D}\left(\Pi_c\right) \right)> 0}$. This implies that the permutation $\Pi_c$ is given. 
As before, for each $1 \leq k \leq n - 1$, let $S^{\Pi_c}_k \in \mathbb{R}^{(k:n)\times (k:n)}$ be the Schur complement after the first $k-1$ steps of column and row pivoting and Gaussian elimination, and let $\Omega^{\Pi_c}_k = \Omega(:, \Pi_r(k:n)) \in \mathbb{R}^{r\times (n-k+1)}$ be the submatrix of $\Omega$ whose columns correspond to the rows of $S^{\Pi_c}_k$. With this notation, we can write
\begin{align*}
{\displaystyle  \mathbb{P} \left( {\cal C}\left(\Pi_c\right) \left| {\cal D}\left(\Pi_c\right) \right. \right)} &= {\displaystyle \mathbb{P} \left(\bigcap_{1 \leq k \leq n} \left(\left( \bigcap_{\substack{ k\leq j\leq n \\ j \neq \alpha_k}} \underline{{\bf C}}_{j,k}^{\Pi_c} \right) \bigcap \overline{{\bf C}}_{\alpha_k,k}^{\Pi_c}\right)\right) } \\
&= 1 - {\displaystyle \mathbb{P}\left( \bigcup_{1 \leq k \leq n} \left(\left( \bigcup_{\substack{ k\leq j\leq n \\ j \neq \alpha_k}} \left(\underline{{\bf C}}_{j,k}^{\Pi_c}\right)^c \right) \bigcup \left(\overline{{\bf C}}_{\alpha_k,k}^{\Pi_c}\right)^c\right)\right) } \\
&\geq 1 - {\displaystyle \sum_{k=1}^n \left( \sum_{\substack{j=k \\ j \neq \alpha_k}}^n \mathbb{P}\left( \left(\underline{{\bf C}}_{j,k}^{\Pi_c}\right)^c \right) + \mathbb{P}\left(\left(\overline{{\bf C}}_{\alpha_k,k}^{\Pi_c}\right)^c\right)\right)  } \\
&\geq 1 - {\displaystyle \sum_{k=1}^n (n-k+1) \exp \left(-\frac{(\epsilon^2 - \epsilon^3)r}{4}\right)} \\
&= 1 - {\displaystyle \frac{n(n+1)}{2} \exp \left(-\frac{(\epsilon^2 - \epsilon^3)r}{4}\right)} \\
\end{align*}
where the third line comes from Lemma \ref{unionb} and the fourth line is from the application of Lemma \ref{randproj}. Plugging this into line (\ref{Eqn:temp1}), we have that
\begin{align*}
\mathbb{P} \left( \bigcup_{\Pi_c} \left( {\cal C}\left(\Pi_c\right) \cap {\cal D} \left( \Pi_c \right) \right)  \right) &\geq \left( 1 - {\displaystyle \frac{n(n+1)}{2} \exp \left(-\frac{(\epsilon^2 - \epsilon^3)r}{4}\right)} \right) \sum_{\Pi_c} \mathbb{P} \left( {\cal D}\left(\Pi_c\right) \right) \\ 
&= 1 - {\displaystyle \frac{n(n+1)}{2} \exp \left(-\frac{(\epsilon^2 - \epsilon^3)r}{4}\right)}
\end{align*}
where the last line is achieved by line (\ref{Eqn:GERCPSucceed}). In order to bound the last line from below by $1-\delta$, we require the ${\displaystyle r > \frac{4}{\epsilon^2-\epsilon^3} \ln \left( \frac{n(n+1)}{2\delta} \right)}$. 
\end{proof}

Below is our main theoretical result on column growth upper bound for GERCP.
\begin{theorem}[Column Growth Factor for GERCP] \label{gercpgf}
Choose $\epsilon, \delta \in (0,1)$ and $g \in (0,1]$. If ${\displaystyle r > \frac{4}{\epsilon^2-\epsilon^3} \ln \left( \frac{n(n+1)}{2\delta} \right) }$, then the pivot growth factor of Algorithm $1.2$ satisfies
\begin{align*}
{\displaystyle\rho_{col}^{gercp}({\bf A})} &\leq {\displaystyle \frac{1}{g^2} \frac{1+\epsilon}{1-\epsilon} \sqrt{e(n+1)} n^{1+\ln\left( g \sqrt{\frac{1+\epsilon}{1-\epsilon}} \right)} n^{\frac{1}{2} \ln \left( n \right)}}
\end{align*}
with probability greater than $1-\delta$, otherwise
\[{\displaystyle \rho_{col}^{gercp}({\bf A}) \leq \frac{1}{\sqrt{n}} 2^{n-1} }\]
\end{theorem}

\begin{remark} What happens if we are extremely unlucky? There is a chance, not exceeding the user-chosen positive probability $\delta$, that at least one column in some Schur complement is not well-preserved. In this case, GERCP could end up picking the wrong column pivot. This, however, is far from a disaster as we still perform the deterministic partial pivoting at every step. Therefore, we can view these randomized column swaps as an ``insurance policy'' against large element growth because with high probability we will attain the growth bounds in Theorem \ref{gercpgf}, but we will definitely attain GEPP growth bounds. \end{remark} 

\begin{proof}
At the $k^{th}$ stage of randomized complete pivoting, our column pivot choice is given by 
\begin{align*}
\ell_k &= \argmax_{k \leq j \leq n} \left\| \Psi_k (:,j) \right\|_2^2 \\
\alpha_k &= \left\{ \begin{array} {cl} \ell_k & \text{ if } g\left\| \Psi_k (:,\alpha_k) \right\|_2 \geq  \left\| \Psi_k(:,k) \right\|_2 \\ k & \text{ otherwise } \end{array} \right.
\end{align*}
Then the row pivot at the $k^{th}$ stage is given as $\beta_k$ where
\begin{equation} \label{alphadef}
\beta_k = \argmax_{k \leq i \leq n} \left| S_k(i,\alpha_k) \right| \end{equation}
At this point, we proceed in a fashion similar to Wilkinson's element growth proof for complete pivoting \cite{wilk1}. Let $p_k = \left| S_k(\beta_k,\alpha_k) \right|$ be the modulus of the pivot element of $S_k$. Also, let $c_k = \left\|S_k(:,\alpha_k) \right\|_2$ be the $\ell_2$-norm of the pivot column of $S_k$. It is important to note that (\ref{alphadef}) along with Lemma \ref{inf_to_2} implies that $c_k \leq \sqrt{n-k+1} p_k$. (i.e. The last line works because $p_k = U(k,k)$ from our a priori pivoting so $p_k$ is an entry of $S_k^{\Pi_c}$.) Then, we have that
\begin{equation} \label{detpslu}
\left| \det \left( S_k^{\Pi_c}(k:m,k:m) \right) \right| = \prod_{j = k}^{m} p_{j} \geq \prod_{j = k}^{m} \frac{1}{\sqrt{n-j+1}} c_j 
\end{equation} 
which holds from the LU decomposition of $S_k^{\Pi_c}$ since the $L$ factor is unit diagonal and the $U$ factor has the $p_j$'s as its diagonal. We can also apply Theorem \ref{hadamard} (Hadamard's Inequality) to the determinant of $S_k^{\Pi_c}(k:m,k:m)$ to get 
\begin{align} 
\left| \det \left( S_k^{\Pi_c}(k:m,k:m) \right) \right| &\leq \prod_{j = k}^{m} \left\| S_k^{\Pi_c}(k:m,j) \right\|_2 \leq \prod_{j = k}^{m} \left\| S_k^{\Pi_c}(:,j) \right\|_2 \nonumber \\
&\leq  \left( \frac{1}{g} \sqrt{\frac{1+\epsilon}{1-\epsilon}} \right)^{m-k} \left\| S_k(:,\alpha_k) \right\|_2^{m-k+1} =  \left( \frac{1}{g} \sqrt{\frac{1+\epsilon}{1-\epsilon}} \right)^{m-k} c_k^{m-k+1} \label{det_eps_bound}
\end{align} 
where the last expression of the first line is achieved by taking the $\ell_2$ norm of the entire $j^{th}$ column instead of the first few entries of the $j^{th}$ column. Also, the second line of the above is achieved from Theorem \ref{Lem:WellPreservedColNrm} by applying (\ref{frac_of_max}) for each $j \neq \alpha_k$. Combining our inequalities (\ref{detpslu}) and (\ref{det_eps_bound}) for $\left| \det \left( S_k(k:m,k:m) \right) \right|$, we get
\[ \left( \frac{1}{g} \sqrt{\frac{1+\epsilon}{1-\epsilon}} \right)^{m-k} \sqrt{\frac{(n-k+1)!}{(n-m)!}} c_k^{m-k+1} \geq \prod_{j = k}^{m} c_{j} \]
for all $1 \leq k \leq m \leq n$. Define $q_k = \ln ( c_k )$. Canceling one $c_k$ on both sides and taking logarithms on both sides, we get
\[ \left(m-k \right) \ln \left( \frac{1}{g} \sqrt{\frac{1+\epsilon}{1-\epsilon}} \right) + \sum_{j=k}^m \ln \sqrt{n-j+1} + \left( m-k \right) q_{k}  \geq \sum_{j=k+1}^{m} q_{j} \]
Dividing by $m-k$, moving each term with any $q_j$ for $k\leq j < m$ to one side and everything else to the other, 
\[ q_{k} - \frac{1}{m-k} \sum_{j=k+1}^{m-1} q_{j} \geq \frac{1}{m-k} q_m - \frac{1}{m-k}\sum_{j=k}^m \ln \sqrt{n-j+1} - \ln \left( \frac{1}{g} \sqrt{\frac{1+\epsilon}{1-\epsilon}} \right) \]
Next, we combine all the inequalities for each $k$ between $1\leq k < m$ to get
\begin{align*} 
\left({\arraycolsep=1.4pt\def\arraystretch{1.5} \begin{array} {cccccc} 1 & -\frac{1}{m-1} & -\frac{1}{m-1} & \cdots & -\frac{1}{m-1} & -\frac{1}{m-1} \\ 0 & 1 & -\frac{1}{m-2} & \cdots & -\frac{1}{m-2} & -\frac{1}{m-2} \\ 0 & 0 & 1 & \cdots & -\frac{1}{m-3} & -\frac{1}{m-3} \\ \vdots & \vdots & \vdots & \ddots & \vdots & \vdots \\ 0 & 0 & 0 & \cdots & 1 & -\frac{1}{2} \\ 0 & 0 & 0 & \cdots & 0 & 1 \end{array} } \hspace{-0.05cm} \right) \hspace{-0.05cm} \left( \hspace{-0.05cm}{\arraycolsep=1.4pt\def\arraystretch{1.5} \begin{array}{c} q_1 \\ q_2 \\ q_3 \\ \vdots \\ q_{m-2} \\ q_{m-1} \end{array} } \hspace{-0.05cm}\right)  \hspace{-0.05cm} &\geq \hspace{-0.05cm} \left( \hspace{-0.1cm} {\arraycolsep=1.4pt\def\arraystretch{1.5} \begin{array}{c} \frac{1}{m-1} q_m \\ \frac{1}{m-2} q_m \\ \frac{1}{m-3} q_m \\ \vdots \\ \frac{1}{2} q_m \\ q_m \end{array} } \hspace{-0.1cm} \right) \hspace{-0.05cm} - \hspace{-0.05cm} \left( \hspace{-0.1cm} \begin{array}{c} \frac{1}{m-1} \ln \sqrt{\frac{n!}{(n-m)!}} \\ \frac{1}{m-2} \ln \sqrt{\frac{(n-1)!}{(n-m)!}} \\ \frac{1}{m-3} \ln \sqrt{\frac{(n-2)!}{(n-m)!}} \\ \vdots \\ \frac{1}{2} \ln \sqrt{\frac{(n-m+3)!}{(n-m)!}} \\ \ln \sqrt{\frac{(n-m+2)!}{(n-m)!}} \end{array} \hspace{-0.1cm} \right) \hspace{-0.05cm} - \hspace{-0.05cm} \left( \hspace{-0.15cm} \begin{array}{c} \ln \left( \frac{1}{g} \sqrt{\frac{1+\epsilon}{1-\epsilon}} \right) \\ \ln \left( \frac{1}{g} \sqrt{\frac{1+\epsilon}{1-\epsilon}} \right) \\ \ln \left( \frac{1}{g} \sqrt{\frac{1+\epsilon}{1-\epsilon}} \right) \\ \vdots \\ \ln \left( \frac{1}{g} \sqrt{\frac{1+\epsilon}{1-\epsilon}} \right) \\ \ln \left( \frac{1}{g} \sqrt{\frac{1+\epsilon}{1-\epsilon}} \right) \end{array} \hspace{-0.15cm} \right) \\ \intertext{and we express the above matrix inequality as}  B \hspace{3.1cm} {\bf q} \hspace{0.5cm} &\stackrel{def}{\geq} \hspace{0.45cm} {\bf v}_1 \hspace{0.53cm} - \hspace{1.4cm} {\bf v}_2 \hspace{1.4cm} - \hspace{1.1cm} {\bf v}_3  \end{align*}
Lemma \ref{matinv} gives us a closed form expression for $B^{-1}$, which happens to be a non-negative matrix. Since $B^{-1}$ only has non-negative entries, the vector inequality above is preserved after multiplying $B^{-1}$ on both sides. To complete this proof, we only need the top row of this vector inequality. Since $q_1$ is the first entry of ${\bf q}$, we also apply ${\bf e}_1^T$ to both sides of the inequality to reduce it to a scalar inequality 
\begin{align} \label{eqn:q1inq} 
q_1 &= {\bf e}_1^T {\bf q} \geq {\bf e}_1^T B^{-1} {\bf v}_1 - {\bf e}_1^T B^{-1} {\bf v}_2 - {\bf e}_1^T B^{-1} {\bf v}_3 \end{align}
Lemma \ref{matinv} also tells us that the first row of $B^{-1}$ is $${\bf e_1}^T B^{-1} = \left[ \begin{array}{cccccc} 1 & \frac{1}{m-1} & \frac{1}{m-2} & \cdots & \frac{1}{3} & \frac{1}{2} \end{array} \right]$$ which we now use to compute/bound ${\bf e}_1^T B^{-1} {\bf v}_1$, ${\bf e}_1^T B^{-1} {\bf v}_2$ and ${\bf e}_1^T B^{-1} {\bf v}_3$:
\begin{align*} 
{\bf e}_1^T B^{-1} {\bf v}_1 &= q_m \left( \frac{1}{m-1} + \sum_{j=1}^{m-2} \frac{1}{(j+1)j} \right) = q_m
\end{align*}
where the last equality was achieved by Lemma \ref{parttele}. Next, we compute ${\bf e}_1^T B^{-1} {\bf v}_2$ which we call the Wilkinson term
\begin{align*}
{\bf e}_1^T B^{-1} {\bf v}_2 &=  \frac{1}{m-1} \sum_{j=1}^m \ln \sqrt{ n -m + j } + \sum_{k= 1}^{m-2} \frac{1}{(k+1)k} \sum_{j=1}^{k+1} \ln \sqrt{ n-m+j } \\
&= \frac{1}{m-1} \sum_{j=1}^m \ln \sqrt{ n-m+j } + \sum_{j=1}^{m-1} \sum_{k=\max\{j-1,1\}}^{m-2} \frac{1}{(k+1)k} \ln \sqrt{ n-m+j } \\
&= \frac{1}{m-1} \ln \sqrt{n} + \sum_{j=1}^{m-1} \left( \frac{1}{m-1} + \sum_{k=\max\{j-1,1\}}^{m-2} \frac{1}{(k+1)k} \right) \ln \sqrt{ n-m+j }   \\
&= \frac{1}{m-1} \ln \sqrt{n} + \sum_{j=1}^{m-1} \frac{1}{\max\{j-1,1\}} \ln \sqrt{ n-m+j } \hspace{1.5cm} \text{(Lemma \ref{parttele})} \\
&= \ln \sqrt{ n-m+1 } + \sum_{j=2}^{m} \frac{1}{j-1} \ln \sqrt{n-m+j} \\
&\stackrel{def}{=} \ln \sqrt{ n-m+1 } + \ln f(m,n-m)
\end{align*}
where we define the \emph{generalized Wilkinson function} to be 
\[f(m,t) \stackrel{def}{=} \sqrt{(2+t)^1 (3+t)^{\frac{1}{2}} (4+t)^{\frac{1}{3}} \cdots (m+t)^{\frac{1}{m-1}} } \]
Next comes ${\bf e}_1^T B^{-1} {\bf v}_3$ or the thresholding term
\begin{equation*}
{\bf e}_1^T B^{-1} {\bf v}_3 = \ln \left( \frac{1}{g} \sqrt{\frac{1+\epsilon}{1-\epsilon}} \right) \sum_{j=1}^{m-1} \frac{1}{j} \leq \left( 1 + \ln(m-1) \right) \ln \left( \frac{1}{g} \sqrt{\frac{1+\epsilon}{1-\epsilon}} \right) 
\end{equation*}
Plugging all of this into (\ref{eqn:q1inq}), we have
\begin{equation*}
q_1 \geq q_m - \ln \sqrt{ n-m+1 } - \ln f(m,n-m) - \left( 1 + \ln(m-1) \right) \ln \left( \frac{1}{g} \sqrt{\frac{1+\epsilon}{1-\epsilon}} \right)
\end{equation*}
Taking the exponential of both sides and rearranging terms, we have that for all $1 \leq m \leq n$
\begin{equation} \label{cratio}
\frac{c_m}{c_1} \leq \sqrt{ n-m+1 }f(m,n-m)\frac{1}{g}\sqrt{\frac{1+\epsilon}{1-\epsilon}}(m-1)^{\ln\left( \frac{1}{g} \sqrt{\frac{1+\epsilon}{1-\epsilon}} \right)}
\end{equation}
Next, we need to relate the ratio $\frac{c_k}{c_1}$ to $\rho_{col}^{gercp}$
\begin{align*}
\rho_{col}^{gercp} &= \frac{\max_k \left\| S_k \right\|_{1,2}}{\left\| S_1 \right\|_{1,2}} \\
&\leq \frac{1}{g} \sqrt{\frac{1+\epsilon}{1-\epsilon}} \max_k \frac{ \left\| S_k(:,\alpha_k) \right\|_{2}}{\left\| S_1 (:,\alpha_1) \right\|_{2}} \hspace{6.41cm} \text{(Lemma \ref{Lem:WellPreservedColNrm})}  \\
&= \frac{1}{g} \sqrt{\frac{1+\epsilon}{1-\epsilon}} \max_k \frac{c_k}{c_1} \\
&\leq \frac{1}{g^2} \frac{1+\epsilon}{1-\epsilon} \max_m \sqrt{ n-m+1 }f(m,n-m)(m-1)^{\ln\left( \frac{1}{g} \sqrt{\frac{1+\epsilon}{1-\epsilon}} \right)} \hspace{2cm} \text{(Eqn (\ref{cratio}))} \\
&\leq \frac{1}{g^2} \frac{1+\epsilon}{1-\epsilon} \max_m \sqrt{e(n-m+2)} (n-m+1) (m-1)^{\ln\left( \frac{1}{g} \sqrt{\frac{1+\epsilon}{1-\epsilon}} \right)} \\
& \hspace{1cm} m^{\frac{1}{4} \ln \left( n \right)} m^{\frac{1}{4} \ln \left( \frac{n}{m} \right)} (n-m+1)^{\frac{1}{4} \ln\left(n-m+1\right)} \\
&\leq \frac{1}{g^2} \frac{1+\epsilon}{1-\epsilon} \max_m \left(\sqrt{e(n-m+2)} (n-m+1) (m-1)^{\ln\left( \frac{1}{g} \sqrt{\frac{1+\epsilon}{1-\epsilon}} \right)} \right) \\
& \hspace{1cm} \max_m \left(m^{\frac{1}{4} \ln \left( n \right)} m^{\frac{1}{4} \ln \left( \frac{n}{m} \right)} \right) \max_m (n-m+1)^{\frac{1}{4} \ln\left(n-m+1\right)} \hspace{1.98cm} \text{(Lemma \ref{wilkbound})} \\
&\leq \frac{1}{g^2} \frac{1+\epsilon}{1-\epsilon} \sqrt{e(n+1)} n^{1+\ln\left( \frac{1}{g} \sqrt{\frac{1+\epsilon}{1-\epsilon}} \right)} \max_m \left(m^{\frac{1}{4} \ln \left( n \right)} m^{\frac{1}{4} \ln \left( \frac{n}{m} \right)} \right) n^{\frac{1}{4} \ln\left(n\right)}
\end{align*}
where the third to last line is from Lemma \ref{wilkbound}. Examining the following derivative
\[ \frac{d}{dm} \ln \left(m^{\frac{1}{4} \ln \left( n \right)} m^{\frac{1}{4} \ln \left( \frac{n}{m} \right)} \right) = \frac{\ln (n)}{m} - \frac{\ln \left(m\right)}{m} \]
which equals zero only when $m=n$, and given the concavity, this point attains the max. Thus, we plug this point in to get
\[ \rho_{col}^{gercp} \leq \frac{1}{g^2} \frac{1+\epsilon}{1-\epsilon} \sqrt{e(n+1)} n^{1+\ln\left( \frac{1}{g} \sqrt{\frac{1+\epsilon}{1-\epsilon}} \right)} n^{\frac{1}{2} \ln \left( n \right)} \]
to get our desired result.
\end{proof}

Next, we slightly add to the statement of the last theorem to add a probabilistic guarantee on the floating point error of the sampling matrix $\Psi$ under the sampling update formula \eqref{Eqn:PsiUpdate}. The floating error analysis of the sampling matrix must be done under the objective that rounding errors cannot corrupt the column lengths by too much, which is suggested by the following result under small probability of failure. This shows that the floating point error cannot grow quickly or exponentially under the given update scheme.
\begin{theorem} [Stability of unconditionally stable sampling update formula] \label{thm:stable}
For Gaussian Elimination with Randomized Complete Pivoting, we have the two guarantees 
\begin{align*} 
\rho_{col}^{gercp}({\bf A}) &\leq \frac{1}{g^2} \frac{1+\epsilon_{JL}}{1-\epsilon_{JL}} \sqrt{e(n+1)} n^{1+\ln\left( g \sqrt{\frac{1+\epsilon_{JL}}{1-\epsilon_{JL}}} \right)} n^{\frac{1}{2} \ln \left( n \right)}
\intertext{and}
\left\| fl\big( \widehat{\Psi}_{k+1} \big) - \widehat{\Psi}_{k+1} \right\|_{1,2} &\leq \epsilon_{mach} \left(1+ \epsilon_{mach} \right)^k \sqrt{kr(1+\epsilon_{JL})} \left(  \sqrt{k} + 5 n \sqrt{n(1+\epsilon_{JL})} \right) \rho_{col}^{gercp} \left(A\right) \left\| A \right\|_{1,2} 
\end{align*}
with probability at least $1 - \frac{n(n+1)}{2}\exp \left(-\frac{(\epsilon_{JL}^2 - \epsilon_{JL}^3)r}{4}\right) - \exp \left(-\frac{\epsilon_{JL}^2rn}{2}\right)$
\end{theorem}
\begin{proof}
First, we apply the relevant norm to both sides to get 
\begin{align}
&\big\| fl\big( \widehat{\Psi}_{k+1} \big) - \widehat{\Psi}_{k+1} \big\|_{1,2} \label{first_ineq_stable} \\ &\leq \epsilon_{mach} \left(1+ \epsilon_{mach} \right)^k \left(\sum_{i=1}^k \left\| \Psi_{i} \left( :, (k+1):n \right) \right\|_{1,2} + 5 \left\| \left| \Omega \right| \left| L_{k+1} \right| \right\|_2 \left\| \left| U_{k+1}(:,(k+1):n) \right| \right\|_{1,2} \right) \nonumber \\
&\leq \epsilon_{mach} \left(1+ \epsilon_{mach} \right)^k \left(\sum_{i=1}^k \left\| \Psi_{i} \left( :, (k+1):n \right) \right\|_{1,2} + 5 \left\| \Omega \right\|_F \left\| L_{k+1} \right\|_F \left\| U_{k+1}(:,(k+1):n) \right\|_{1,2} \right) \nonumber
\end{align}
where we bound the $2$-norm of the entrywise absolute value by the frobenius norm to get the last line. Next, Theorem \ref{gercpgf} and the definition of the event $\overline{{\bf C}}_{\alpha_k,k}^{\Pi_c}$ from \eqref{eqn:Cupper} give us both our desired bound on $\rho_{col}^{gercp} \left( {\bf A} \right)$ and 
\begin{equation} \left\| \Psi_k \right\|_{1,2} = \left\| \Psi_k (:,\alpha_k) \right\|_{2} \leq \sqrt{r(1+\epsilon_{JL})} \left\| S_k \right\|_{1,2} \leq \sqrt{r(1+\epsilon_{JL})} \rho_{col}^{gercp} \left({\bf A}\right) \left\| {\bf A} \right\|_{1,2} \label{faew} \end{equation}
with probability of failure bounded above by $\frac{n(n+1)}{2}\exp \left(-\frac{(\epsilon_{JL}^2 - \epsilon_{JL}^3)r}{4}\right)$. Since GERCP is a top-heavy method from Definition \ref{topheavy}, we have $\|L_{k+1}\|_F \leq \sqrt{kn}$ because $L_{k+1}$ is a $n\times k$ matrix with each entry bounded above by $1$ in absolute value. Also, Theorem \ref{colback} gives us that $\|U_{k+1}\|_{1,2} \leq n \rho_{col}^{gercp} \left( {\bf A} \right) \left\| {\bf A } \right\|_{1,2}$. Finally, we observe that the map $f(\Omega) = \left\| \Omega \right\|_F$ is a lipschitz map with Lipschitz constant $L = 1$ and apply Theorem \ref{thm:lip} to get
\[ \mathbb{P} \left\{ \left\| \Omega \right\|_F \geq \sqrt{rn} + t \right\} \leq e^{-\frac{t^2}{2}} \]
where we applied Jensen's inequality to get $\mathbb{E} \left\| \Omega \right\|_F \leq \sqrt{\mathbb{E} \left\| \Omega \right\|_F^2 } = \sqrt{rn}$ from Proposition $10.1$ from \cite{Tropp}. Set $t = \epsilon_{JL} \sqrt{rn} $. Then plug this and the inequality \eqref{faew} into the first inequality \eqref{first_ineq_stable} to arrive at our desired conclusion with a union bound.
\end{proof}

\section{Numerical Experiments}

\subsection{Block GERCP}
Algorithm \ref{Alg:gercp} was presented in a form for ease of presentation of Theorem \ref{gercpgf}. For efficient numerical implementation, we need to develop a block version of GERCP in Algorithm \ref{Alg:gercp} to increase locality and memory/cache re-use with more BLAS-3 calls. Algorithm \ref{Alg:brgercp} below is styled after dgetf2.f and dgetrf.f for block GEPP in LAPACK with double precision floating point numbers. Instead of increasing the pivot-block size from $p=1$, we introduce a loop-blocksize parameter $b\geq 1$. It repeatedly performs $b$ steps of randomized column pivoting and partial pivoting followed by a blocksize $b$ Schur complement update. 
\begin{algorithm}\label{Alg:brgercp}{\bf Block GERCP} \\
\headerrule

\begin{tabular}{ll}
{\bf Input:} & $n \times n$ matrix $A$, sampling dimension $r>0$, block size $b$ \\
{\bf Output:} & lower triangular $L$, upper triangular $U$, row permutation $\Pi_r$, column permutation $\Pi_c$.  \\
\end{tabular}

\noindent\headerrule \\
{\bf sample} $\Omega(i,j) \sim \mathcal{N}(0,1)$ for all $1\leq i\leq r$ and $1\leq j\leq n$ \\
{\bf compute} $\Psi = \Omega A$ \\
{\noindent \bf for} $\underline{k} = 1:b:n - 1$ {\bf do}
\begin{INDENT}
{\bf set} $\overline{k}=\underline{k}+\min\{b,n-\underline{k}+1\}-1$ \\
{\bf for} $k = \underline{k}:\overline{k}$ {\bf do} \begin{INDENT}
\begin{enumerate}
\item {\bf compute} ${\displaystyle \alpha = \left\{ \begin{array}{ll} \argmax_{k \leq j \leq n} \left\| \Psi(:,j) \right\|_2 & \text{ if } n-k \geq r \\ \argmax_{k \leq j \leq n} \left\| A(k:n,j) \right\|_2 & \text{ otherwise } \end{array} \right. }$. \\
 {\bf swap} columns $k$ and $\alpha$ of $A$, $\Psi$ and $\Omega$ (*).
\item {\bf compute} ${\displaystyle \beta = \argmax_{k \leq i \leq n} \left| A(i,k) \right|}$. \\
{\bf swap} rows $k$ and $\beta$ of $A$ (*).
\item {\bf compute} $A(k\!+\!1\!:\!n,k) = A(k\!+\!1\!:\!n,k)/A(k,k)$;
\item {\bf compute} $A(k\!+\!1\!:\!n,k\!+\!1\!:\!\overline{k}) = A(k\!+\!1\!:\!n,k\!+\!1\!:\!\overline{k}) - A(k\!+\!1\!:\!n,k) * A(k,k\!+\!1\!:\!\overline{k}) $;
\item {\bf compute} $A(k,\overline{k}\!+\!1\!:\!n) = A(k,\overline{k}\!+\!1\!:\!n) - A(k,\underline{k}\!:\!k\!-\!1) * A(\underline{k}\!:\!k\!-\!1,\overline{k}\!+\!1\!:\!n)$; 
\item {\bf update} $\Psi(:,k\!:\!n)$ with Algorithm \ref{Alg:update}
\end{enumerate} 
\end{INDENT}
{\bf end for}
\item {\bf compute} $A(\overline{k}\!+\!1\!:\!n,\overline{k}\!+\!1\!:\!n) = A(\overline{k}\!+\!1\!:\!n,\overline{k}\!+\!1\!:\!n) - A(\overline{k}\!+\!1\!:\!n,\underline{k}\!:\!\overline{k}) * A(\underline{k}\!:\!\overline{k},\overline{k}\!+\!1\!:\!n) $;

\end{INDENT}
{\bf end for}\\
\headerrule
\end{algorithm}

The main work of Algorithm \ref{Alg:brgercp} is in the last step, the repeated computation of the matrix $A(\overline{k}\!+\!1\!:\!n,\overline{k}\!+\!1\!:\!n) $. The outer loop for $\underline{k}$ is similar to the main loop in dgetrf.f, while the inner loop for $k$ is similar to the main loop in dgetf2.f. The main modifications occur on lines $7, 8$ and the BLAS-3 Schur complement update after the end of the inner loop. Line $7$ updates the $U$ factor so that we can use it to update the sketching matrix on line $8$. As in dgetf2.f, the inner loop for $k$ is designed to work on a tall-skinny matrix, whereas the outer loop for $\underline{k}$ performs fast BLAS-3 updates on the rest of the matrix. For practical reasons, we stop using the sampling matrix once the dimensions of the Schur complement become less than or equal to the sampling dimension $r>0$. It is important to note that the proof of Theorems \ref{gercpgf} and \ref{thm:stable} is easily modified to apply to this versions of the algorithm with the exact same guarantees on element growth.  We present the procedure for updating the Schur sampling matrix $\Psi$. 

\begin{algorithm}\label{Alg:update}{\bf Update procedure for Schur sampling matrix $\Psi$} \\
\headerrule

\begin{tabular}{ll}
{\bf Input:} & $r \times n$ matrix $\Psi$, $n \times n$ working matrix $A$, $r \times n$ random matrix $\Omega$ \\
{\bf Output:} & $r \times m$ matrix $\Psi$  \\
\end{tabular}

\noindent\headerrule \\
{\noindent \bf if} pivot $\left| A(k,k) \right| \geq \sqrt{\epsilon_{mach}} \left\| \Psi_1 \right\|_{1,2}$ {\bf then} apply Eqn \eqref{Eqn:PsiUpdateFast} with \\
\begin{INDENT}
$\Psi(:,(k+1):n) \longleftarrow {\displaystyle \Psi(:,(k+1):n) - \frac{\Psi(:,k) A(k,(k+1):n)}{A(k,k)}}$
\end{INDENT}
{\noindent \bf else} apply Eqn \eqref{Eqn:PsiUpdate} with \\
\begin{INDENT}
$\Psi(:,(k+1):n) \longleftarrow {\displaystyle \Psi(:,(k+1):n) - \left[\Omega(:,k) + \Omega(:,(k+1):n) A((k+1):n,k)\right] A(k,(k+1):n)  }$
\end{INDENT}
\headerrule
\end{algorithm}

\begin{remark} The first updating formula in Algorithm \ref{Alg:update} is slightly more efficient than the second one. While we have not observed it in our numerical computations, it potentially could lead to inaccurate column selections for some highly ill-conditioned matrices in pathological cases. Our numerical experiments also suggest that the execution time of Algorithm \ref{Alg:update} is typically a relatively small fraction of the total execution time of of GERCP. Thus, one might want to use the second updating formula in Algorithm \ref{Alg:update} for a more robust numerical implementation. \\
\end{remark} 

\subsection{Numerical Results} We ran our experiments on two different machines. The runtime results were preformed on a single node of NERSC's Carver machine with two quad-core Intel Xeon X5550 2.67 GHz processors and 24 GB of RAM. The rest of the numbers are generated on a laptop with an Intel i7-3632QM CPU and 8GB of RAM. All of the code here was run using the MKL BLAS. We used the open source Netlib version of GEPP. Our version of GERCP was obtained by modifying the Netlib GEPP Fortran code. This allows for an easy and fair comparison between GEPP and GERCP by insuring that the version of GEPP used to compare against has similar cache optimizations. It is worth noting that the Intel MKL version of GEPP is much faster than both Netlib GEPP and GERCP because of superior cache optimizations. Figure \ref{fig:runtimes1} from our runtime experiments shows that as the matrix size increases, the percent difference in runtime decreases to a negligible amount. This agrees with our theory, which tells us the $O(n^2)$ operations required to maintain the sampling matrix and pivot columns does not grow as quickly as the $O(n^3)$ operations required to actually factor the matrix $A$. Even when the relative time difference is high, the absolute time difference is unnoticeable for a single factorization as shown in by the run times for $N=3000$ in Table \ref{tab:runtab} below.
\begin{figure}[!h]
\centering
\includegraphics[scale=0.9]{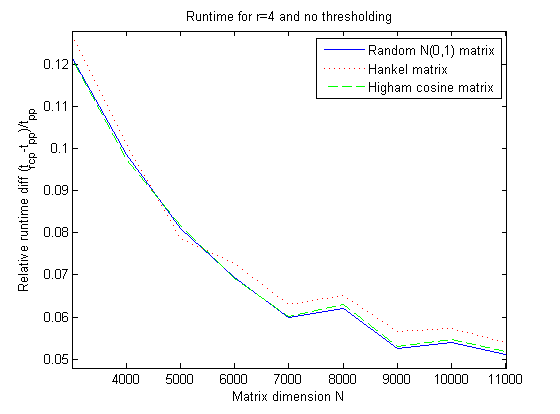}
\caption{Comparing the run times of GERCP and GEPP Fortran code each averaged over $10$ different trials}
\label{fig:runtimes1}
\end{figure}

\begin{table} \label{tab:runtab}
\begin{center}
\begin{tabular}{c|c|c|c|c|c} 
N & 3000    & 5000   & 7000 & 9000 & 11000 \\
\hline
$t_{rcp}$ (secs) & 2.250   & 9.623  & 25.506  & 53.607 & 96.662  \\ \hline
$t_{pp}$ (secs)  & 2.006   & 8.902  & 24.066  & 50.934 & 91.967 \\ \hline
$\frac{t_{rcp}-t_{pp}}{t_{pp}}$ & 12.20\% & 8.10\% & 6.00\%  & 5.20\% & 5.10\% \\
\end{tabular}
\end{center}
\caption{Average run times of GERCP and GEPP over 10 seperate trials.}
\end{table}

In section \ref{sec:gepp}, we reviewed stability issues associated with to most commonly used Gaussian elimination pivoting strategy, partial pivoting. Now, we produce numerical experiments showing the improved stability properties of GERCP. The Wilkinson, Generalized Wilkinson and Volterra matrices that we use in Figures \ref{fig:elemgrow1} and \ref{fig:errnorm1} in these plots are as described in Section \ref{sec:gepp}. The Wilkinson-type matrices serve as the worst case matrix for GEPP instability with entries that grow exponentially, where the standard Wilkinson matrix has the quickest exponential growth with a base of $2$. This is exemplified by the dashed blue and red lines in the log-log plots of element growth and backwards error of GEPP in Figures \ref{fig:elemgrow1} and \ref{fig:errnorm1}. However, GERCP fixes this by impeding element growth with column pivots to leave the backwards error near machine precision. As far as Gaussian elimination goes, the most pathological examples, that we characterize as \emph{passive aggressive} element growth, include the Volterra matrix. These matrices exhibit just enough element growth, but no more, to cause an unacceptable level of backwards error. In contrast, the element/column growth for the Wilkinson-type matrices is so massive that the problem is trivial to detect and fix for any GE algorithm with both row and column pivots to correct. This case too is effortlessly corrected by GERCP as shown by our experiments.
\begin{figure}[!h]
\centering
\includegraphics[scale=0.9]{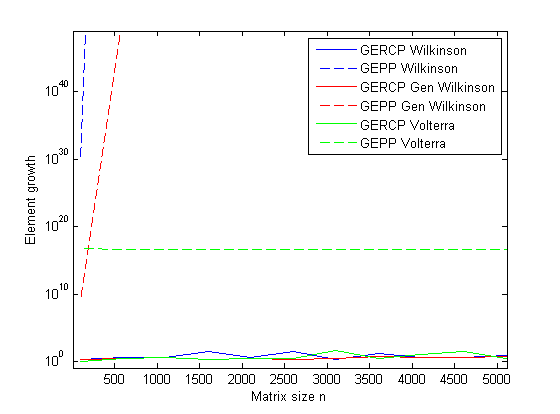}
\caption{Element growth for diabolical matrices with GERCP and GEPP Fortran code}
\label{fig:elemgrow1}
\end{figure}
\begin{figure}[!h]
\centering
\includegraphics[scale=0.9]{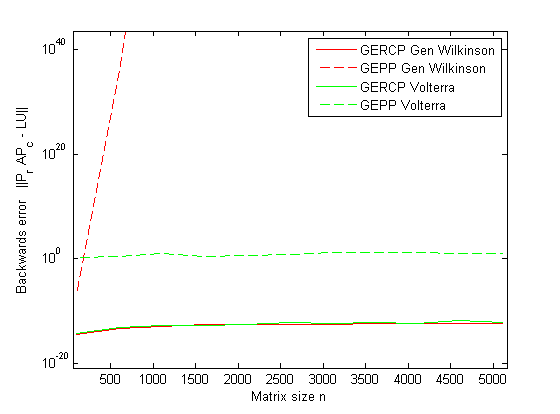}
\caption{Backwards error for diabolical matrices with GERCP and GEPP Fortran code}
\label{fig:errnorm1}
\end{figure}

\subsection{Backward error for random linear systems} Suppose we wish to solve the linear system ${\bf A}{\bf x} = {\bf b}$ where ${\bf A}\in\mathbb{R}^{n\times n}$ iid $\mathcal{N}(0,1)$ standard normal and ${\bf b} \in \mathbb{R}^n$ iid $\mathcal{N}(0,1)$ standard normal. This system is known to be well conditioned \cite{ChenDongarra} and the LU factorization is stable under partial pivoting and complete pivoting \cite{trefethenschreiber}. We measure the accuracy of a linear solve with the \emph{relative residual} \[ \frac{\left\| {\bf A} \widehat{\bf x} - {\bf b} \right\|_{\infty}}{\left\| {\bf A} \right\|_{\infty} \left\| \widehat{\bf x}\right\|_{\infty}}\]
In Figure \ref{fig:sysresid}, we plot the relative residual error for GERCP with different sampling parameters $r>0$, along with competing methods like GEPP, GECP, GERP (rook) and GE2CP. While complete pivoting consistently obtained the smallest residual, GERCP, GERP and GE2CP all produced similar relative residuals which were clearly better than GEPP and not much worse than that of GECP. This shows that GERCP produces a better quality solution than GEPP even when GEPP is given a well-conditioned system. The Figure \ref{fig:sysresid} suggests that the relative residual for the random normal linear system is improved by almost a factor of $2$. This savings becomes more important when you work with smaller precession floating point numbers like single or half precession floating point numbers. These smaller precession floating point numbers are becoming commonly used on co-processing platforms like GPGPUs as result in dramatic run time improvements. Also, for different linear systems this improvement in the relative residual can be much higher as in the case of the Wilkinson-type and Volterra matrices.
\begin{figure}[t!] \label{fig:sysresid}
	\centering
	\includegraphics[scale=0.9]{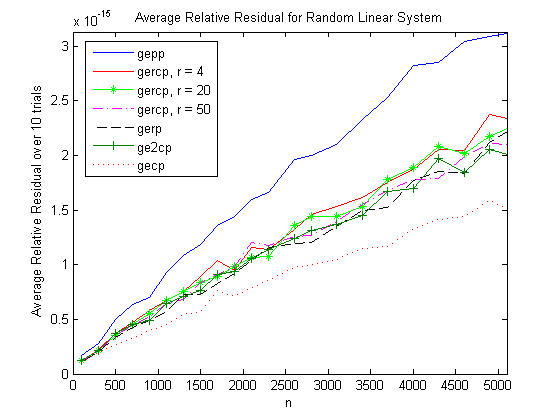}
	\caption{Average Relative Residual over $10$ trials. This suggests that GERCP should improve the relative residual of a linear solve over GEPP by at least half the improvement that GECP would provide.}
\end{figure}

We also look at the element growth within the LU factors for different pivoting strategies. In \cite{trefethenschreiber}, Trefethen and Schreiber study the element growth factors for GEPP and GECP on standard normal matrices. They conjecture that $\mathbb{E} \left( \rho_n^{gepp} \right) \approx n^{2/3}$ and $\mathbb{E} \left( \rho_n^{gecp} \right) \approx n^{1/2}$. We plot the element growth for GEPP and GECP along with the element growth for GERP (rook), GE2CP and GERCP with different values of the sampling parameter $r>0$ in Figure \ref{fig:syseg}.

\begin{figure}[t!] \label{fig:syseg}
	\centering
	\includegraphics[scale=0.9]{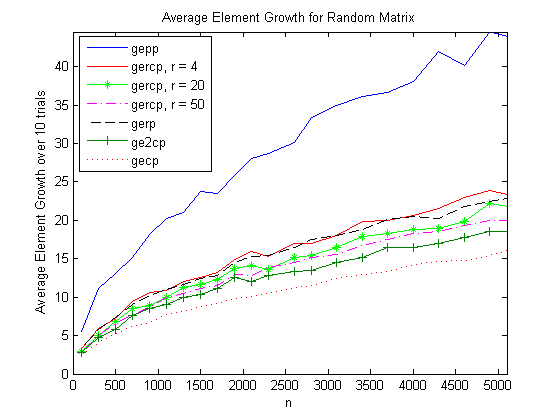}
	\caption{Average element growth of iid standard Normal random matrix over $10$ trials.}
\end{figure}

{\bf Acknowledgements and Future Work.} The authors would like to thank James Demmel and Laura Grigori for many helpful and fruitful discussions. There are many directions for future research into randomized complete pivoting. One direction is to study and provide guarantees for variants of GERCP that allow for a large pivot-block size $p>1$. The authors believe that such a method could be much faster due to an increase in BLAS-$3$ operations, and could also provide rank-revealing style guarantees for low-rank approximations generated by incomplete LU factorizations \cite{miranian2003strong}. It will also be important to develop a cache optimized version if this code to be competitive to Intel MKL LAPACK GEPP. Another important avenue for future research is to use these techniques to make communication avoiding tournament pivoting LU (CALU) more stable by adding randomized column pivots \cite{donfack2014dynamically,grigori2011calu}.

\bibliographystyle{plain}
\bibliography{total}

\section{Appendix}
\begin{lemma} [Partial fraction telescoping sum] \label{parttele}
Let $q>r>0$ be positive integers. Then we have
\[ \frac{1}{r} = \frac{1}{q} + \sum_{j=r}^{q-1} \frac{1}{(j+1)j} \]
\end{lemma}
\begin{proof}
Let $s\in \mathbb{N}$ such that $r\leq s < q$. We use common denominators to subtract the two following fractions
\[ \frac{1}{s} - \frac{1}{s+1} =  \frac{1}{s(s+1)} \]
Then we take the telescoping sum to arrive at our result
\begin{align*}
 \frac{1}{r} - \frac{1}{q} &= \sum_{s=r}^{q-1} \frac{1}{s} - \frac{1}{s+1} \\
&= \sum_{s=r}^{q-1} \frac{1}{s(s+1)}
\end{align*}
\end{proof}

\begin{lemma} [Special Matrix Inverse] \label{matinv}
Let $B = \left( b_{ij}\right)_{1\leq i,j \leq n} \in \mathbb{R}^{n\times n}$ be an upper triangular matrix given by 
\[b_{ij} = \left\{ \begin{array} {cl} 0 & \text{ if } i > j \\ 1 & \text{ if } i=j \\ -\frac{1}{i} & \text{ if } i < j \end{array} \right.\]
or in other words
\[B = \left( \begin{array} {cccccc} 1 & -\frac{1}{n} & -\frac{1}{n} & \cdots & -\frac{1}{n} & -\frac{1}{n} \\ 0 & 1 & -\frac{1}{n-1} & \cdots & -\frac{1}{n-1} & -\frac{1}{n-1} \\ 0 & 0 & 1 & \cdots & -\frac{1}{n-2} & -\frac{1}{n-2} \\ \vdots & \vdots & \vdots & \ddots & \vdots & \vdots \\ 0 & 0 & 0 & \cdots & 1 & -\frac{1}{2} \\ 0 & 0 & 0 & \cdots & 0 & 1 \end{array} \right) \]
Then the inverse $B^{-1} = \left( c_{ij}\right)_{1\leq i,j \leq n} \in \mathbb{R}^{n\times n}$ is given by
\[ c_{ij} = \left\{ \begin{array} {cl} 0 & \text{ if } i > j \\ 1 & \text{ if } i=j \\ \frac{1}{n-j+2} & \text{ if } i < j \end{array} \right. \] 
or 
\[ B^{-1} = \left( \begin{array} {cccccc} 1 & \frac{1}{n} & \frac{1}{n-1} & \cdots & \frac{1}{3} & \frac{1}{2} \\ 0 & 1 & \frac{1}{n-1} & \cdots & \frac{1}{3} & \frac{1}{2} \\ 0 & 0 & 1 & \cdots & \frac{1}{3} & \frac{1}{2} \\ \vdots & \vdots & \vdots & \ddots & \vdots & \vdots \\ 0 & 0 & 0 & \cdots & 1 & \frac{1}{2} \\ 0 & 0 & 0 & \cdots & 0 & 1 \end{array} \right) \]
\end{lemma}
\begin{proof}
Rewrite $B$ into the product of elementary matrices or atomic triangular matrices
\begin{align*} 
B &= \scalemath{0.9}{\left( \begin{array} {cccccc} 1 & 0 & 0 & \cdots & 0 & 0 \\ 0 & 1 & 0 & \cdots & 0 & 0 \\ 0 & 0 & 1 & \cdots & 0 & 0 \\ \vdots & \vdots & \vdots & \ddots & \vdots & \vdots \\ 0 & 0 & 0 & \cdots & 1 & -\frac{1}{2} \\ 0 & 0 & 0 & \cdots & 0 & 1 \end{array} \right) \cdots \left( \begin{array} {cccccc} 1 & 0 & 0 & \cdots & 0 & 0 \\ 0 & 1 & -\frac{1}{n-1} & \cdots & -\frac{1}{n-1} & -\frac{1}{n-1} \\ 0 & 0 & 1 & \cdots & 0 & 0 \\ \vdots & \vdots & \vdots & \ddots & \vdots & \vdots \\ 0 & 0 & 0 & \cdots & 1 & 0 \\ 0 & 0 & 0 & \cdots & 0 & 1 \end{array} \right) \left( \begin{array} {cccccc} 1 & -\frac{1}{n} & -\frac{1}{n} & \cdots & -\frac{1}{n} & -\frac{1}{n} \\ 0 & 1 & 0 & \cdots & 0 & 0 \\ 0 & 0 & 1 & \cdots & 0 & 0 \\ \vdots & \vdots & \vdots & \ddots & \vdots & \vdots \\ 0 & 0 & 0 & \cdots & 1 & 0 \\ 0 & 0 & 0 & \cdots & 0 & 1 \end{array} \right)} \\
&\stackrel{def}{=} \hspace{1.7cm} M_2 \hspace{1.7cm} \cdots \hspace{2.4cm} M_{n-1} \hspace{4.3cm} M_n
\end{align*}
Taking the inverse, we see that
\begin{align*}
B^{-1} &= \hspace{2.1cm} M_n^{-1} \hspace{3.9cm} M_{n-1}^{-1} \hspace{2.1cm} \cdots \hspace{2cm} M_2^{-1} \\
&= \left( \begin{array} {cccccc} 1 & \frac{1}{n} & \frac{1}{n} & \cdots & \frac{1}{n} & \frac{1}{n} \\ 0 & 1 & 0 & \cdots & 0 & 0 \\ 0 & 0 & 1 & \cdots & 0 & 0 \\ \vdots & \vdots & \vdots & \ddots & \vdots & \vdots \\ 0 & 0 & 0 & \cdots & 1 & 0 \\ 0 & 0 & 0 & \cdots & 0 & 1 \end{array} \right) \left( \begin{array} {cccccc} 1 & 0 & 0 & \cdots & 0 & 0 \\ 0 & 1 & \frac{1}{n-1} & \cdots & \frac{1}{n-1} & \frac{1}{n-1} \\ 0 & 0 & 1 & \cdots & 0 & 0 \\ \vdots & \vdots & \vdots & \ddots & \vdots & \vdots \\ 0 & 0 & 0 & \cdots & 1 & 0 \\ 0 & 0 & 0 & \cdots & 0 & 1 \end{array} \right) \cdots \left( \begin{array} {cccccc} 1 & 0 & 0 & \cdots & 0 & 0 \\ 0 & 1 & 0 & \cdots & 0 & 0 \\ 0 & 0 & 1 & \cdots & 0 & 0 \\ \vdots & \vdots & \vdots & \ddots & \vdots & \vdots \\ 0 & 0 & 0 & \cdots & 1 & \frac{1}{2} \\ 0 & 0 & 0 & \cdots & 0 & 1 \end{array} \right)
\end{align*}
We leave computing this product as an exercise to the reader. Lemma \ref{parttele} will be useful.
\end{proof}

\begin{lemma} [Generalized Wilkinson function bound] \label{wilkbound}
Let $f(m,t) := \sqrt{(2+t)(3+t)^{\frac{1}{2}}\cdots(m+t)^{\frac{1}{m-1}}}$ be the generalized Wilkinson function. Then, we have
\begin{align}
f(m,t) &\leq \sqrt{e(t+2)(t+1)} m^{\frac{1}{4} \ln \left( m+t \right)} m^{\frac{1}{4} \ln \left( \frac{m+t}{m} \right)} (t+1)^{\frac{1}{4} \ln\left(t+1\right)} \label{eqnwilk2}
\end{align}
\end{lemma}
\begin{proof}
We have 
\[ \ln \left( f^2 (m,t) \right) = \sum_{k=2}^m \frac{1}{k-1} \ln \left( t+k \right) = \sum_{k=1}^{m-1} \frac{1}{k} \ln \left( t+k+1 \right) \]
Observe the identity $\frac{1}{k}-\frac{1}{k+t+1} = \frac{t+1}{k(k+t+1)}$ and note that the function $\frac{1}{k} \ln \left( t+k+1 \right)$ is decreasing in $k$. We use the integral approximation along with integration by parts to get (\ref{eqnwilk2})
\begin{align*}
\ln \left( f^2 (m,t) \right) &= \ln \left( t+2 \right) + \sum_{k=2}^{m-1} \frac{1}{k} \ln \left( k+t+1 \right) \\
&\leq \ln \left( t+2 \right) + \int_{1}^{m-1} \frac{1}{x} \ln \left( x+t+1 \right) dx \\
&= \ln \left( t+2 \right) + \ln(x) \ln (t+x+1) \bigg|_{1}^{m-1} - \int_{1}^{m-1} \frac{1}{x+t+1} \ln \left(x\right) dx \\
&= \ln \left( t+2 \right) + \ln(x) \ln (t+x+1) \bigg|_{1}^{m-1} - \int_{1}^{m-1} \frac{1}{x} \ln \left(x\right) dx + \int_{1}^{m-1} \frac{t+1}{x(x+t+1)} \ln \left(x\right) dx \\
&= \ln \left( t+2 \right) + \ln(m-1) \ln (m+t) - \frac{1}{2} \ln^2 \left(m-1\right) + \int_{1}^{m-1} \frac{t+1}{x(x+t+1)} \ln \left(x\right) dx \\
&\leq \ln \left( t+2 \right) + \ln(m) \ln (m+t) - \frac{1}{2} \ln^2 \left(m\right) + \int_{1}^{\infty} \frac{t+1}{x(x+t+1)} \ln \left(x\right) dx \\
&= \ln \left( t+2 \right) + \frac{1}{2} \ln(m) \ln (m+t) + \frac{1}{2} \ln \left(m\right) \ln \left(\frac{m+t}{m}\right) + \int_{1}^{\infty} \frac{t+1}{x(x+t+1)} \ln \left(x\right) dx
\end{align*}
where the second to last line follows from the fact that
\[\frac{d}{dx} \ln(x) \ln (m+t) - \frac{1}{2} \ln^2 \left(x\right) = \frac{\ln (m+t)}{x} - \frac{\ln \left(x\right)}{x} \geq 0 \]
for all $x \leq m+t$. The inequality (\ref{eqnwilk2}) follows from lemma \ref{impropint}.
\end{proof}

\begin{lemma} [Useful inequality for improper integral] \label{impropint}
We have the following inequality
\begin{align}
\int_{1}^{\infty} \frac{t+1}{x(x+c)} \ln \left(x\right) dx &\leq \frac{1}{2} \ln^2 \left(c\right) + \ln \left( c \right) + 1
\end{align}
\end{lemma}
\begin{proof}
\begin{align*}
\int_{1}^{\infty} \frac{c}{x(x+c)} \ln \left(x\right) dx 
&= \int_{1}^{c} \frac{c}{x(x+c)} \ln \left(x\right) dx + \int_{c}^{\infty} \frac{c}{x(x+c)} \ln \left(x\right) dx \\
&\leq \int_{1}^{c} \frac{1}{x} \ln \left(x\right) dx + \int_{c}^{\infty} \frac{c}{x^2} \ln \left(x\right) dx \\
&= \frac{1}{2} \ln^2 \left(x\right) \bigg|_{1}^{c} - \frac{c}{x} \left( \ln \left( x\right) + 1 \right) \bigg|_{c}^{\infty} \\
&=  \frac{1}{2} \ln^2 \left(c\right) + \ln \left( c \right) + 1
\end{align*}
\end{proof}

Lastly, we include a technical theorem used in the proof of Theorem \ref{thm:stable}
\begin{theorem} [Concentration of measure for Lipschitz functions of a Gaussian matrix \cite{Bogdanov}] \label{thm:lip}
Suppose that $f(x)$ is a Lipschitz function on matrices:
\[ \left| f(A)-f(B) \right| \leq L_f \left\| A - B \right\|_F \qquad \text{for all } A,B \in \mathbb{R}^{m\times n} \]
Sample a matrix $G\in \mathbb{R}^{m\times n}$ with independent standard Gaussian $\mathcal{N}(0,1)$ entries. Then, for all $t\geq 0$,
\begin{align*} \mathbb{P} \left\{ f\left( G \right) \geq \mathbb{E} \left[ f\left(G\right)\right] + L_f t \right\} &\leq e^{-\frac{t^2}{2}} 
\end{align*}
\end{theorem}

\end{document}